\documentclass[12pt,a4paper]{article}
\usepackage{amsmath}
\usepackage{graphicx}
\usepackage{epsfig}%
\usepackage{amsfonts}%
\usepackage{amssymb}
\usepackage{algpseudocode}						
\usepackage{algorithm}	
\usepackage{hhline}

\usepackage{xcolor}
\usepackage{mathtools}
\usepackage{diagbox}						
\usepackage{subfigure}
\usepackage{nicefrac}
\usepackage{multirow}
\usepackage{ulem}

\newcommand{\cn}{\color{black}}

\newtheorem{theorem}{Theorem}[section]

\newtheorem{lemma}[theorem]{Lemma}

\newtheorem{remark}[theorem]{Remark}

\newenvironment{proof}[1][Proof]{\noindent \emph{#1.} }{\hfill \
\rule{0.5em}{0.5em}}

\makeatletter\@addtoreset{equation}{section}\makeatother
\makeatletter\@addtoreset{figure}{section}\makeatother
\makeatletter\@addtoreset{table}{section}\makeatother
\newcommand{\dotprod}[2]{\langle#1,#2\rangle} 					
\newcommand{\R}{{\mathbb{R}}}       							
\newcommand{\nell}{{n_{(\ell)}}}
\newcommand{\aplus}{a_{i+\frac{1}{2}}}
\newcommand{\aminus}{a_{i-\frac{1}{2}}}

\newcommand{\uu}{\mathbf{u}}
\newcommand{\y}{\mathbf{y}}

\newcommand{\Xo}{\mathbf{X}^{(0)}}

\begin{document}

\title{Tensor numerical method for optimal control problems constrained by an 
elliptic operator  with general  rank-structured  coefficients}
 
 \author{
        Boris N. Khoromskij\thanks{Max-Planck-Institute for
        Mathematics in the Sciences, Inselstr.~22-26, D-04103 Leipzig,
        Germany ({\tt bokh@mis.mpg.de}); Trier University, 
        FB IV - Department of Mathematics, D-54296, Trier, Germany.} 
         \and
         Britta Schmitt\thanks{Trier University,
        FB IV - Department of Mathematics, D-54296, Trier,
        Germany ({\tt schmittb@uni-trier.de}).}
        \and        
        Volker Schulz\thanks{Trier University,
        FB IV - Department of Mathematics, D-54296, Trier,
        Germany ({\tt volker.schulz@uni-trier.de}).}
        }


\maketitle
\vspace*{-6mm}
\begin{abstract}
We introduce tensor numerical techniques for solving optimal control problems  
constrained by elliptic operators in $\mathbb{R}^d$, $d=2,3$, with variable coefficients, 
which can be represented in a low rank separable form. 
We construct a preconditioned iterative method with an adaptive rank truncation 
for solving the equation for the control function, 
governed by a sum of the elliptic operator and its inverse $M=A + A^{-1}$, both discretized 
over large $n^{\otimes d}$, $d=2,3$, spatial grids. Two basic solution schemes are proposed and analyzed.
In the first approach, one solves iteratively the initial linear system of equations with the matrix $M$
such that the matrix vector multiplication with the elliptic operator inverse, $y=A^{-1} u,$
is performed as an embedded iteration by using a rank-structured solver for the 
equation of the form $A y=u$. The second numerical scheme avoids the embedded iteration by reducing the 
initial equation to an equivalent one with the polynomial system matrix of the form $A^2 +I$. 
For both schemes, a low Kronecker rank spectrally equivalent preconditioner 
is constructed by using the corresponding matrix valued function of the anisotropic Laplacian diagonalized in the 
Fourier basis.
Numerical tests for control problems in 2D setting confirm the linear-quadratic 
complexity scaling of the proposed method in the univariate grid size $n$.  Further, 
we numerically demonstrate that 
for our low rank solution method, a cascadic multigrid approach reduces the number of PCG iterations
considerably, however the total CPU time remains merely the same as for the unigrid iteration.  
\end{abstract}

\section{Introduction}\label{Int:SH} 

Optimization problems that are constrained by partial differential equations (PDEs) have a long 
  history in mathematical literature since they arise in a wide range of applications  in 
  different fields of natural science,  see  for example \cite{Troeltzsch:2005,Allaire:07,HerKun:2010} 
  for some comprehensive examples and related references. 
  Being studied for many years, 
  tracking-type problems which trace the discrepancy between the solution of the PDE and a given 
  target state represent a very important class of optimal control problems \cite{Reyes:2015}.
 Among others, a particular class of applications is related to structural topology optimization. 
  In such problems the discretization and numerical treatments of the elliptic PDE in constraints, 
  that determines the relation between the optimal design $y$ and control function $u$,
  \begin{equation} \label{eqn:Constraint eqn}
  {A} y:= \mbox{div} \, a(x) \mbox{grad} \, y= u, \quad x\in \mathbb{R}^d,
   \end{equation}
  can be performed by the traditional FEM methods dealing with sparse matrices. 
  In cases when the entire problem can be reduced to the explicit equation for the control function, the system 
  matrix includes the weighted sum of the initial operator $A$ and its nonlocal inverse 
  $A^{-1}$, $A + A^{-1}$. The numerical solution 
  of the resultant equation in multidimensional setting becomes the challenging problem due to 
  the treatment of the fully populated matrix $A^{-1}$.
  
   In this way, the multigrid methods for elliptic equations proved to be the efficient option 
   in the context of control problems since 
   they apply to the case of rather general equation coefficients, while the 
numerical complexity scales linearly in the number of grid points in the computational 
domain in $\mathbb{R}^d$, see \cite{BoSch:2009,BoSch:2012}. However, the multigrid approach
for equations with nonlocal operators of the form $A + A^{-1}$ does not provide the uniform 
convergence rate since both the smooth and highly oscillating eigen-functions correspond to large 
eigenvalues, $\lambda + \lambda^{-1}$. In this situation we propose to use a PCG iteration with a
 spectrally equivalent preconditioner of low Kronecker rank (K-rank) applied to the system matrix $A + A^{-1}$.

In the recent decade, tensor numerical methods proved to provide an efficient strategy 
for large scale computer simulations of PDE driven mathematical models \cite{Hack_Book12,KhorBook:18,Khor2Book2:18}. 
In particular, the basic numerical multilinear algebra algorithms for the Tucker, canonical, TT and
QTT tensors have been described in 
\cite{DMV-SIAM2:00,HaKhtens:04I,Hack_Book12,KhorBook:18} (see also references therein).
Iterative methods for solving linear systems in the rank-structured tensor formats have been 
considered in particular in \cite{KreTob:2010,DolOs:2011}.

Recently the numerical analysis of optimal control problems constrained by space-fractional 
elliptic operators was intensively discusses in the literature 
\cite{AntilOtarola:2015,HLMMV:2018,DPSS:2014,HKKS:18,SchmittKh2Sch:20,AntilDraGreen:2020,Bonito:18,Reyes:2015},
see also papers on the discretization and analysis of fractional PDEs \cite{DuLazPas:18,HarLazarov:20} 
and \cite{FracLapl:2018,Kwasn:17,Schwab:18}.

In the present paper 
we introduce tensor numerical techniques for solving optimal control problems  
constrained by elliptic operators in $\mathbb{R}^d$, $d=2,3$, with variable coefficients, 
which can be represented in low rank separable form. 
We construct a preconditioned iterative method with an adaptive rank truncation 
for solving the equation for the control function, 
governed by a sum of the elliptic operator and its inverse $M=A + A^{-1}$, both discretized 
over large $n^{\otimes d}$, $d=2,3$, spatial grids. 
The low-rank approximation of matrix-valued functions and the adaptive rank-truncation in PCG 
iteration for 3D case can be performed by the multigrid Tucker approximation scheme \cite{KhKh3:08}, while 
for 2D problems we use the standard reduced SVD.
Two basic solution schemes are proposed and analyzed.

In the first approach, one can solve iteratively the initial linear system of equations 
with the system matrix $M$
such that the matrix vector multiplication with the elliptic operator inverse, $y=A^{-1} u$,
is performed by solving the equation $A y=u$ via the embedded PCG iteration 
by using algorithms based on multilinear operations on rank-structured data. 
The second numerical scheme avoids the embedded iteration by transforming the initial 
equation to the equivalent one with a polynomial system matrix of the form {$AM=A^2 +I$}. 
For both approaches, a low Kronecker rank spectrally equivalent preconditioner 
is constructed by using the reciprocal matrix valued function of the anisotropic Laplacian diagonalized in the 
Fourier basis.  Throughout this work, we refer to the 
\textit{anisotropic Laplacian} whenever we use a scaled version of the classic Laplace  
operator $\Delta$ of the form $A= \sum_{\ell=1}^d a_\ell \frac{\partial^2}{\partial x_\ell^2}$, 
where the positive  constants $a_1,\ldots,a_d $ may vary on different scales. 
Numerical examples confirm the log-linear complexity scaling of the proposed method in the univariate grid size $n$.
 For a large univariate grid size $n$ the additional complexity term of the order of $O(n^2)$ can be observed, which is due to the cost 
of matrix-vector multiplication. 

 In order to numerically enhance the pcg iteration scheme, we make use of a cascadic multigrid approach, 
that means the iterative pcg method on the grid of size $n\times n$ starts with an interpolation of the solution from the 
previous grid $\frac{n}{2} \times \frac{n}{2}$ and no coarse grid correction is done, see \cite{Shaidur:95,BornDeuf:1996} for more details.

The rest of the paper is organized as follows. In Section \ref{sec:RankR_oper} we describe in detail
the low Kronecker rank discrete form of elliptic operators with  rank-$R$ separable coefficients
and then set up the equation for the control function. Section \ref{sec:MatrTensRepr} 
presents the construction and analysis of low rank spectrally equivalent preconditioners. 
In Section \ref{sec:numerics} we discuss the computational complexity. 
 In particular, we present numerical study of the presented tensor-based solver applied to control problems
constrained by the 2D elliptic equations with variable coefficients. 
Section \ref{sec:Conclusions} concludes the paper.
  
 \section{Low rank discrete form of elliptic operators with  rank-$R$ separable coefficients}
 \label{sec:RankR_oper}

 In the following, we introduce the target class of tracking-type optimal control problems 
 that motivates our work. 
 Subsequently, the discretization of the target elliptic operator and the effect of separable 
 variables on the structure 
 of the system matrix are presented. Finally, at the end of the chapter, the equation for the 
 problem-specific control is given and analyzed.

 Given the design function $y_\Omega \in L^2(\Omega)$ on $\Omega :=(0,1)^d$, $d=1,2,3$,
first, we consider the optimization problem for the cost functional
\begin{equation} \label{eqn:cost_func}
 \min_{y,u} J(y,u):=\frac{1}{2}\int_\Omega (y(x) -y_\Omega(x))^2\, dx + \frac{\gamma}{2} 
 \int_\Omega u^2(x) \,dx, \quad \gamma >0, 
\end{equation}
constrained by the elliptic boundary value problem  
in $\Omega$ for the state variable  $y \in H_0^1(\Omega)$,
\begin{equation} \label{eqn:basic_setting}
  {\cal A} y  := -\nabla^T \cdot \mathbb{A}(x)\nabla y =  u, \quad x\in \Omega,\;
  u\in L_2(\Omega),
  \end{equation}
endorsed with the homogeneous Dirichlet boundary conditions 
on $\Gamma = \partial \Omega $, i.e., $y_{|\Gamma}=0$.
 The diagonal $d\times d $ coefficient matrix takes the form 
$\mathbb{A}(x)= \mbox{diag}\{a(x),...,a(x)\}\in \mathbb{R}^{d\times d}$. 
 
 We consider optimal control problems with  an elliptic operator in the constraint
 equation (\ref{eqn:basic_setting})
 that has rank-$1$  separable variable coefficients in the form
\begin{equation} \label{eqn:Oper_Rank1}
 \mathcal{A} =- \sum_{\ell=1}^d \frac{d}{d x_\ell} 
\prod_{m=1}^d a_{m}(x_m) \frac{d}{d x_\ell} ,\quad x\in (0,1)^d,
\end{equation}
and the generalization to the rank-$R$ variable coefficients  leading to 
\begin{equation} \label{eqn:Oper_Rank1}
 \mathcal{A} =- \sum_{\ell=1}^d   \frac{d}{d x_\ell} 
 \left(\sum_{k=1}^R \prod_{m=1}^d a^{(k)}_{m}(x_m)\right)\frac{d}{d x_\ell} ,\quad x\in (0,1)^d.
\end{equation}

\subsection{Galerkin stiffness matrix in tensor product basis}\label{ssec:Galerkin} 

To enhance the application of the Kronecker product structures, 
we apply the FEM Galerkin scheme to the equation (\ref{eqn:Constraint eqn})
by means of the tensor-product piecewise linear finite elements 
$$
\{\psi_{{\bf i}}(x):=\psi_{i_1}(x_1) \cdots \psi_{i_d}(x_d)\}, 
\quad {{\bf i}}=(i_1,\ldots,i_d), \quad i_\ell\in {\mathcal I}_\ell=\{1,\ldots,n_\ell\}, 
$$
$\ell=1,\ldots,d$, where $\psi_{i_\ell}(x_\ell)$ are the univariate piecewise linear hat 
 functions\footnote{Notice that the univariate grid size $n_\ell$ 
designates the total problem size $ N_d = n_1 n_2\cdots n_d$.
}. 
The unknown optimal design and control functions  $y$ and $u$ 
in equation (\ref{eqn:Constraint eqn}) are represented by ${y(x)=\sum_i y_i \psi_{{\bf i}}(x)} $ and $u(x)=\sum_i u_i \psi_{{\bf i}}(x) $,
with the coefficient vector ${\bf y}=\{y_i \} \in \mathbb{R}^{N_d}$
and ${\bf u}=\{u_i \} \in \mathbb{R}^{N_d}$, respectively.

The $N_d\times N_d$ stiffness matrix in the FEM Galerkin discretization of 
the elliptic operator  (\ref{eqn:Oper_Rank1}) is constructed by the standard mapping of the multi-index $ {\bf i}$
into the long univariate index $1\leq i \leq N_d$ for the active degrees of freedom. 
For instance,  we use the so-called big-endian convention for $d=3$ and $d=2$
\[
 {{\bf i}}\mapsto i:= i_3 + (i_2-1)n_3 + (i_1-1)n_2 n_3, 
 \quad {{\bf i}}\mapsto i:= i_2 + (i_1-1)n_2,
\]
respectively. In what follows,  we consider the case $d=3$ in more detail. 

In our discretization scheme,  we calculate the stiffness matrix  
by assembling the Kronecker product terms by using representation of the equation
coefficient ${a}(x)$ as an $R$-term sum of separable functions
\begin{equation} \label{eqn:Coef_RankR}
 a(x_1,x_2,x_3) =  \sum\limits^{R}_{k=1} a^{(1)}_k (x_1) a^{(2)}_k (x_2) a^{(3)}_k (x_3), 
\end{equation}
where the functions of univariate arguments, $a^{(\ell)}_k (x_\ell)$, are supposed to be positive.

To that end, let us assume for the moment that the scalar diffusion coefficient $a(x_1,x_2,x_3) $ can be 
represented in the separate form (rank-$1$ representation)
\[
 a(x_1,x_2,x_3) = a^{(1)} (x_1) a^{(2)} (x_2) a^{(3)} (x_3) >0,
\]
so that the target elliptic operator takes form
\begin{align}\nonumber 
 {\cal A}= & -a_2(x_2) a_3(x_3)\frac{d}{d x_1}a_1(x_1)\frac{d}{d x_1} -  
   a_1(x_1) a_3(x_3)\frac{d}{d x_2}a_2(x_2) \frac{d}{d x_2} -  
   a_1(x_1) a_2(x_2)\frac{d}{d x_3}a_3(x_3)\frac{d}{d x_3}. 
 \end{align}
Then the entries of the Galerkin 
stiffness matrix $A=[a_{{\mu}{\nu}}]\in \mathbb{R}^{N_d\times N_d}$ can be represented by (denote $J=[0,1]$)
\begin{align} \label{eqn:tensStMatr}
a_{{\mu}{\nu}}   &=   \langle {\mathcal A} \psi_{\mu}, \psi_{\nu} \rangle =
\int_\Omega  a^{(1)} (x_1) a^{(2)} (x_2) a^{(3)} (x_3) \nabla \psi_\mu \cdot \nabla \psi_\nu  d x \nonumber \\
    & = \int_{J}  a^{(1)} (x_1) \dfrac{\partial \psi_{\mu_1} (x_1)}{\partial x_1 } 
    \dfrac{\partial \psi_{\nu_1} (x_1)}{\partial x_1 } d x_1   
   \int_{J} a^{(2)} (x_2) \psi_{\mu_2}(x_2)\psi_{\nu_2}(x_2) d x_2  
   \int_{J} a^{(3)} (x_3) \psi_{\mu_3}(x_3)\psi_{\nu_3}(x_3) d x_3 \nonumber\\
   & +  \int_{J} a^{(1)} (x_1) \psi_{\mu_1}(x_1) \psi_{\nu_1} (x_1)d x_1  
   \int_{J} a^{(2)} (x_2) \dfrac{\partial \psi_{\mu_2} (x_2)}{\partial x_2 } 
   \dfrac{\partial \psi_{\nu_2} (x_2)}{\partial x_2 } d x_2 
   \int_{J} a^{(3)} (x_3) \psi_{\mu_3}(x_3)\psi_{\nu_3}(x_3) d x_3 \nonumber\\
   & +  \int_{J} a^{(1)} (x_1) \psi_{\mu_1}(x_1) \psi_{\nu_1} (x_1)d x_1 
    \int_{J} a^{(2)} (x_2) \psi_{\mu_2}(x_2)\psi_{\nu_2}(x_2) d x_2
    \int_{J} a^{(3)} (x_3) \dfrac{\partial \psi_{\mu_3} (x_3)}{\partial x_3 } 
    \dfrac{\partial \psi_{\nu_3} (x_3)}{\partial x_3 } d x_3.
  \end{align} 
The discretized constraints equation (\ref{eqn:Constraint eqn}) then takes the form
\[
 A{\bf y}= M {\bf u}, \quad \mbox{with the mass matrix} \quad 
 M=[\langle \psi_{i}, \psi_{j} \rangle]_{i,j=1}^{N_d}\in \mathbb{R}^{N_d\times N_d}.
\]
This FEM discretization scheme leads to the low Kronecker rank structure in the stiffness matrix 
as described in the following.

 \subsection{Rank structured Kronecker form of the system matrix}\label{ssec:KronRepr}

The equation (\ref{eqn:tensStMatr}) ensures the rank-3 Kronecker product representation 
of the stiffness matrix
\[
 {A} = {A}_1 \otimes S_2\otimes S_3 + S_1 \otimes {A}_2 \otimes S_3 + S_1 \otimes S_2\otimes {A}_3 ,
\]
where $\otimes$ denotes the conventional Kronecker product of matrices. 

Here ${A}_1=[a_{\mu_1\nu_1}]\in \mathbb{R}^{n_1\times n_1} $,  and 
${A}_2=[a_{\mu_2 \nu_2}]\in \mathbb{R}^{n_2\times n_2}$  and ${A}_3=[a_{\mu_3 \nu_3}]\in \mathbb{R}^{n_3\times n_3}$
denote the univariate (tri-diagonal) stiffness matrices, while
 ${S_1=[s_{\mu_1\nu_1}]\in \mathbb{R}^{n_1\times n_1}}$, 
${S_2=[s_{\mu_2 \nu_2}]\in \mathbb{R}^{n_2\times n_2}}$ and  ${S_3=[s_{\mu_3 \nu_3}]\in \mathbb{R}^{n_3\times n_3}}$ 
define the weighted mass matrices, for example
{\small 
\[
 a_{\mu_1\nu_1}= 
 {\int_{(0,1)} } a^{(1)} (x_1) 
    \frac{\partial \psi_{\mu_1} (x_1)}{\partial x_1 }   
   \frac{\partial \psi_{\nu_1} (x_1)}{\partial x_1 } d x_1 , \quad
  s_{\mu_1\nu_1}= 
  \int_{(0,1)} a^{(1)} (x_1) \psi_{\mu_1}(x_1) \psi_{\nu_1} (x_1)d x_1 .
\]\
We use the integration scheme as in \cite{SchmittKh2Sch:20} to obtain the symmetric tridiagonal 
stiffness matrices $A_\ell$, $\ell=1,2,3$  with entries given by 

 $$ A_\ell = \dfrac{1}{h^2} \begin{bmatrix}
a_{11}^\ell & a_{12}^\ell & & & &\\
a_{21}^\ell & a_{22}^\ell & a_{23}^\ell & & & \\
 & &&\ddots&&\\
 &&&&\ddots&a_{(n-1),n}^\ell\\
&&& &a_{n,(n-1)}^\ell & a_{n,n}^\ell
\end{bmatrix}, \ 
 a_{i,j} = \begin{cases}
 \aplus & j-i = 1\\
 \aminus & j-i = -1\\
 -\aplus - \aminus & i = j,\quad i \not\in \{1, \nell\} \\
 -\aplus & i = j = 1\\
 -\aminus & i = j = \nell.
 \end{cases}$$

Without loss of approximation accuracy $O(h^2)$, we can apply the lumping procedure that substitutes 
the tri-diagonal mass matrices by diagonal ones, $S_\ell \mapsto D_\ell$, 
such that the entries of the diagonal matrix $D_\ell=\mbox{diag}\{d_1,\ldots,d_{n_\ell}\}$, 
$\ell=1,2,3$, are defined as the row sum of the elements in the initial mass matrix $S_\ell$,
\[
 d_i= s_{i,i-1} + s_{i,i} + s_{i,i+1}, \quad i=1,\ldots,n_\ell, \quad s_{1,0}=s_{n_\ell,n_\ell+1}=0.
\]
Then the corresponding stiffness matrix takes the form of Kronecker rank-$3$ representation,
\[
 A = A_1 \otimes D_2 \otimes D_3 +  D_1 \otimes A_2 \otimes D_3 + D_1 \otimes D_2 \otimes A_3,
\]
where $D_\ell$, $\ell=1,2,3$, denotes the diagonal matrices defined above.
In the case of rank-$R$ separable coefficients, the stiffness matrix obeys the form of Kronecker 
rank-$3R$ structure as following
\begin{equation} \label{eqn:3D_operator}
 A = \sum\limits^{R}_{k=1}
 \left( A_{1,k} \otimes D_{2,k} \otimes D_{3,k} +  D_{1,k} \otimes A_{2,k} \otimes D_{3,k} + 
 D_{1,k} \otimes D_{2,k} \otimes A_{3,k}\right).
\end{equation}

Now given the matrix $A$, in the case of space-fractional control, one might be interested in the 
fast rank-structured solution of an equation with a fractional elliptic operator of non-local type 
(arising in applications to optimal control problems)
\[
 (A^\alpha + A^{-\alpha}) u = F, \quad \alpha \in (0,1],
\]
in the spirit of the recent papers \cite{HKKS:18,SchmittKh2Sch:20}.
However, for the expected level of generality concerning the equation coefficients in (\ref{eqn:Coef_RankR}), 
we are not able to diagonalize the matrix $A$ 
in a separable tensor product eigen-basis of one-dimensional elliptic operators as it was  described
in \cite{HKKS:18,SchmittKh2Sch:20}. That is due to the fact that the diagonal matrices $D_\ell$ do not commute any more 
with the univariate stiffness matrices $A_\ell$.
On the other hand, the full-format representation of non-local operators $A^{\pm \alpha}$ 
becomes non-tractable for 3D discretizations on fine enough grids.
Due to these observations, in what follows, we consider the case $\alpha= 1$. 

The general model with $\alpha \in (0,1]$, i.e. for factional control, 
can be treated by the tensor based methods where the discretization scheme can be 
based on the contour integral representation, see for example 
\cite{GaHaKh3:02,HaHighTrefeth:08,Higham_MatrFunc:08}. This topic will be considered elsewhere. 

 In order to present an efficient tensor-vector multiplication $Ax$ where $A$ is given in Kronecker 
rank-$3R$ structure presented in \eqref{eqn:3D_operator}, let $x \in \R^N$ be a vector given 
in low-rank format, i.e. $$ x \approx \sum_{j = 1}^S x_1^{(j)} \otimes x_1^{(j)}\otimes x_1^{(j)}, $$ 
with vectors $x_\ell^{(j)} \in \R^n_\ell$ and $S \ll \min(n_1,n_2,n_3).$ Then, a tensor-vector product can be computed by 
{\small
\begin{align} \label{eqn:multischeme}
& \qquad  Ax \nonumber \\
&\approx   \left( \sum_{k = 1}^R  (A_1^{(k)} \otimes M_2^{(k)} \otimes M_3^{(k)} + M_1^{(k)} \otimes A_2^{(k)} \otimes M_3^{(k)} +  M_1^{(k)} \otimes M_2^{(k)} \otimes A_3^{(k)})\right) \left(\sum_{j = 1}^S u_1^{(j)} \otimes u_2^{(j)} \otimes u_3^{(j)}\right) \nonumber \\
\begin{split}
&=   \sum_{k = 1}^R \sum_{m = 1}^S ( A_1^{(k)}u_1^{(m)} \otimes M_2^{(k)}u_2^{(m)} \otimes M_3^{(k)}u_3^{(m)} + M_1^{(k)}u_1^{(m)} \otimes A_2^{(k)}u_2^{(m)} \otimes M_3^{(k)}u_3^{(m)}\\
& \hspace*{1.8cm}  + M_1^{(k)}u_1^{(m)} \otimes M_2^{(k)}u_2^{(m)} \otimes A_3^{(k)}u_3^{(m)}  ).
\end{split}
\end{align}
}
Aforementioned product \eqref{eqn:multischeme} can be calculated in factorized form in $O(d^2RSn^2)$ flops, where $n = \max(n_1,n_2,n_3).$

\subsection{Setting up the equation for control and analysis of the condition number}\label{ssec:KronRepres} 

We  describe necessary first order conditions by a modification of the 
constructions in \cite{HKKS:18}.
 In the following discussion, we use the notations $A_{FE}$ and $A_{FD}$ for the finite element and finite 
difference matrices, respectively, discretizing the elliptic operator $\mathcal{A}$. Here the finite element matrix $A=A_{FE}$ 
is constructed as described in \S \ref{ssec:KronRepr}, while the finite difference matrix  $A_{FD}$ is obtained from the latter by 
rescaling, see (\ref{eqn:FD_FEscaling}).

We consider a version of the control problem \eqref{eqn:cost_func} constrained by \eqref{eqn:basic_setting}, 
discretized  by a FE method  on a uniform grid in each dimension as described above,
\begin{align} \label{eqn:DiscrOptimal}
	\min_{\mathbf{y},\mathbf{u}} =  \ &\frac{1}{2}(\mathbf{y} - \mathbf{y}_\Omega)^TM(\mathbf{y} 
	- \mathbf{y}_\Omega)  + \frac{\gamma}{2} \mathbf{u}^T M \mathbf{u}\\
	\text{s.\,t.}\ ~ A_{FE} \mathbf{y} = \ &  M\mathbf{u},\label{eqn:DiscrOptimal2}
\end{align}
where the coefficient vectors $\y, \y_\Omega, \uu \in \R^N$ denote the discretized state $y$, design $y_\Omega$ 
and control $u$, respectively. 
 The matrix $M$ will be a mass matrix in the finite element case. 
   To derive the FD scheme, we simplify the mass matrix to $M= h^d I$ without loss of 
the asymptotic approximation rate.
Then we arrive at the following scaling relation 
\begin{equation}\label{eqn:FD_FEscaling}
A_{FD}= M^{-1}  A_{FE}= h^{-d} A_{FE} 
\end{equation}
 in the finite difference case. This translates the constraint equation (\ref{eqn:DiscrOptimal2}) to the 
form $A_{FD} \mathbf{y} =  \mathbf{u}$, while in the equation (\ref{eqn:DiscrOptimal}) the factor $M$ can 
be skipped  since the scaling by constant does not effect the minima of the functional.

Now setting up the Lagrangian function with the help of the discrete adjoint variable $\mathbf{p}$, 
\begin{equation*}
	L(\mathbf{y},\mathbf{u},\mathbf{p}) = \frac{1}{2}(\mathbf{y} - 
	\mathbf{y}_\Omega)^T (\mathbf{y} - \mathbf{y}_\Omega)  
	+ \frac{\gamma}{2} \mathbf{u}^T  \mathbf{u} + \mathbf{p}^T(A_{FD} \mathbf{y} -  \mathbf{u}),
\end{equation*}
and differentiating it with respect to all three variables, we end up with the system of equations 
\begin{equation*}
	\begin{bmatrix}
		E & O &  A_{FD}\\
		O & \gamma M & - E\\
		A_{FD} & - E & O
	\end{bmatrix}
	\begin{pmatrix}
		\mathbf{y}\\
		\mathbf{u}\\
		\mathbf{p}\\
	\end{pmatrix} =
	\begin{pmatrix}
		 \mathbf{y}_\varOmega\\
		\mathbf{0}\\
		\mathbf{0}\\
	\end{pmatrix} \
	\begin{matrix}
	\mathrm{(I)} \\
	\mathrm{(II)} \\
	\mathrm{(III)} \\
	\end{matrix}.
\end{equation*}
This system contains the necessary first order optimality conditions for 
the  minimizers  solving  the discretized  
optimal control problem (\ref{eqn:DiscrOptimal}),  (\ref{eqn:DiscrOptimal2}).

Then the state equation $\mathrm{(III)}$ can be solved for $\y$, yielding
\begin{equation}\label{eqn:State_FE}
	\mathbf{y} = A_{FD}^{-1}  \mathbf{u}.
\end{equation}
Subsequently to solving $\mathrm{(II)}$ for $\mathbf{p}$, the adjoint equation $\mathrm{(I)}$ 
eventually provides an equation for the control $\mathbf{u}$, that is
\begin{equation} \label{eqn:Lagrange_cont_FD}
\big( A_{FD}^{-1}   + \gamma  A_{FD} \big) \mathbf{u}  = \mathbf{y}_\varOmega.
\end{equation}

The straightforward preconditioned iteration for solution of the initial equation 
for the control function,
\begin{equation}\label{eqn:Control}
 (\gamma A + A^{-1}) u = F,
\end{equation}
needs the calculation of the matrix times vector multiplication $y_k= A^{-1} u_k$ for each current 
iterand $u_k$, which is equivalent to solving the linear system of equations $ A y_k =u_k$. 
The latter equation can be solved by 
the PCG iteration with rank truncation at the optimal complexity scaling of the order of $O(n \log^q n)$. 
However, the two-level embedded iteration might be too expensive in the case when the inner iteration 
requires a rather accurate solution. 

To avoid the embedded iteration, we propose to solve the equivalent modified equation
\begin{equation}\label{eqn:Control_mod}
 (A^2 +I ) u = A F.
\end{equation}
In this case we only need the matrix vector multiplication with the squared matrix $A^2$ and 
a spectrally equivalent preconditioner for $A^2 +I $.
The latter is constructed by using a low Kronecker rank  approximation of the inverse 
matrix $(\Delta^2 +I)^{-1} $ by using the factorization in the 3D 
Fourier basis
\[
 (\Delta^2 +I)^{-1}= \mathcal{F}^* ( \varLambda^2 +1)^{-1} \mathcal{F},
\]
which can be done by the similar techniques as in \cite{HKKS:18,SchmittKh2Sch:20}. 
From now on, $\Delta $ denotes the discrete Laplacian on a tensor grid
and $\varLambda$ is the diagonal matrix of the corresponding eigenvalues.

Finally, for a spacing $h \in (0,1)$ we notice that the condition numbers of both matrices $A + A^{-1}$ and $A^2 +I$
are of the order of $O(h^{-1})$ and $O(h^{-2})$, respectively, for the optimal scaling of the matrix $A$.
In the following Lemma we consider two different scalings for the target matrix $A$.
\begin{lemma}\label{lem:condMatr} 
 (A) Assume that $\sigma(A)\in [h,1/h]$, then the condition number estimates hold
 \[
  \mbox{cond}(A+A^{-1}) = O(h^{-1}),
 \]
 \[
  \mbox{cond}(A^2+ I) =O(h^{-2}).
 \]
(B) Under the assumption $\sigma(A)\in [1,1/h^2]$ we obtain
\[
  \mbox{cond}(A+A^{-1}) = O(h^{-2}),
 \]
 \[
  \mbox{cond}(A^2+ I) =O(h^{-4}).
 \]
\end{lemma}
\begin{proof}
We consider the upper and lower bounds of the respective spectral functions
 \[
  F_1(\lambda)= \lambda + \lambda^{-1}\quad \mbox{and}\quad F_2(\lambda)= \lambda^2 +1, \quad \lambda\in \sigma(A).
 \]
 For example, consider the function $F_1(\lambda)$ in case (A). In this scenario it holds  $\mbox{min}(F_1(\lambda))=2$, while 
 the maximal value, $O(h^{-1})$, is achieved on both endpoints of the spectrum $\sigma(A)$. Applying the similar arguments to 
 other cases completes the proof.
\end{proof}

In view of Lemma \ref{lem:condMatr}, a spectrally close preconditioner is mandatory for the efficient solution 
of both equations (\ref{eqn:Control}) and (\ref{eqn:Control_mod}).
In the following section, we construct and analyze the low $K$-rank
spectrally close preconditioners for the system matrices $A + A^{-1}$ and $A^2 +I$.

\section{Preconditioning issues}\label{sec:MatrTensRepr} 

\subsection{Rank-structured spectrally close preconditioners: condition number estimates}\label{ssec:Precond} 

We are interested in the construction of spectrally equivalent preconditioners for the matrices 
 $B=A + A^{-1}$ and $S= A^2 +I$  in 
(\ref{eqn:Control}) and (\ref{eqn:Control_mod}) in the form of low $K$-rank matrices. For both cases
the resultant preconditioner is based on the simple spectrally close preconditioners constructed for 
the initial stiffness matrix $A$.

\subsubsection{Preconditioning the stiffness matrix $A$ }

For the ease of presentation, first, we describe the construction of preconditioners for the two dimensional case, 
such that the target stiffness matrix is presented in the rank structured form
\begin{equation} \label{eqn:2D_operator}
 A = \sum\limits^{R}_{k=1}
 \left( A_{1,k} \otimes D_{2,k}  +  D_{1,k} \otimes A_{2,k}\right), 
\end{equation}
where  $D_{\ell,k}=\mbox{diag} \{ {\bf d}_{\ell,k}\} >0$ are positive diagonal matrices and the tridiagonal matrices 
$A_{\ell,k}$ are each spectrally equivalent to the 1D discrete Laplacian $\Delta$ with 
equivalence constants $a_{\ell,k}^{\pm}>0$, $\ell=1,2$,  that is it holds 
$$ a_{\ell,k}^{-} \Delta \le A_{\ell,k} \le a_{\ell,k}^{+} \Delta, \quad \ell = 1,2, \quad k = 1,\dots, R. $$ The constants $a_{\ell,k}^{\pm}>0$, $\ell=1,2$, $k = 1,\dots, R$ can be assumed to be a majorant and minorant to the respective coefficients $a_\ell^k(x_\ell)$, $\ell=1,2$, $k = 1,\dots, R$ as part of \eqref{eqn:Oper_Rank1}.

For preconditioning needs,
we simplify the diagonal matrices by weighted identity matrices, 
where positive weights $d_{\ell,k}^0$ are calculated as the average values of the corresponding 
diagonal vectors ${\bf d}_{\ell,k}=\mbox{diag}\{D_{\ell,k}\}\in \mathbb{R}^n$,
\begin{equation}\label{eqn:weights}
d_{\ell,k}^0=n^{-1}\langle {\bf d}_{\ell,k},1 \rangle, 
\end{equation}
and introduce the constants 
\[
a_{\ell,k}^{0} =\frac{1}{2}(a_{\ell,k}^{+} + a_{\ell,k}^{-}) > 0,\quad k=1,\ldots, R.
\] 
Now agglomerating the anisotropy factors in both dimensions
\begin{equation}\label{eqn:ani_factors}
a_{\ell}^{0} = \sum\limits^{R}_{k=1}a_{\ell,k}^{0} \prod\limits_{m = 1; m \neq \ell}^d d_{m,k}^0, \quad \ell=1,2, 
 \end{equation}
we construct the simple preconditioner $B_1^{-1}$ with $B_1 = A_1 + A_1^{-1}$ by using the anisotropic Laplacian
\begin{equation}\label{eqn:PrecA1_2D}
 A_1= a_{1}^{0} \Delta_1 \otimes I_2 + a_{2}^{0} I_1 \otimes \Delta_2,
 \end{equation}
 where  $\Delta_1 = \Delta_2 = \Delta$, 
so that the storage and matrix-vector multiplication with preconditioner $B_1^{-1}$ can be 
implemented in $O(n \log n)$ operations by low $K$-rank tensor decomposition of $B_1^{-1}$ in 
the Fourier basis, see \cite{SchmittKh2Sch:20} for details.  Likewise, the matrix $A_1$ can be also used for the construction of 
preconditioners $S_1^{-1}$ with $S_1=A_1^2 +I$ applied to the modified system matrix $S$. 

The more advanced preconditioner $B_2^{-1}$ with $B_2 = A_2 + A_2^{-1}$, as well as $S_2^{-1}= (A_2^2+I)^{-1}$,
based on the elliptic operator inverse with variable coefficients, see \cite{SchmittKh2Sch:20}, are both constructed by
using the matrix
\begin{equation}\label{eqn:PrecA2_2D}
 A_2= A_1^0 \otimes I_2 + I_1 \otimes A_{2}^{0},
 \end{equation}
where
\[
 A_1^0 = \sum\limits^{R}_{k=1} d_{2,k}^0 A_{1,k}, \quad
 A_2^0 = \sum\limits^{R}_{k=1} d_{1,k}^0 A_{2,k}.
\]
The matrix in (\ref{eqn:PrecA2_2D}) can be applied to enhance the convergence of PCG iteration in the situation with highly variable coefficients. 
In this case, following \cite{SchmittKh2Sch:20},
we represent the low $K$-rank tensor decomposition of $B_2^{-1}$
by using the factorization of $B_2$ in the eigenbasis of the one-dimensional elliptic operators $A_1^0$ and $A_2^0$.
The asymptotic complexity estimate of the preconditioner includes the term of the order 
of $C_0 \, n^2$ with a small constant $C_0$ in front of, which, does not fit (formally) the concept of 
$O(n \log n)$-complexity  tensor based solver.
However, in numerical tests we observe that the expense of matrix-vector multiplication 
operation of the complexity $n^2$ remains negligible 
for moderate grid-size $n$ and becomes noticeable only for the very large $n$.

The 3D analogies of both preconditioners 
 $B_1= A_1 + A_1^{-1}$ and $B_2= A_2 + A_2^{-1}$ as well as $S_1=A_1^2 +I$ and $S_2= A_2^2 +I$  
are based  on the following constructions of matrices $A_1$ and $A_2$. 
Given the stiffness matrix $A$ in (\ref{eqn:3D_operator}),
the  representations for both   $B_1$ and 
$B_2 $,  as well as for $S_1$ and $S_2$,  follow from the explicit expansions for the spectrally close to $A$ matrices 
$A_1$, and $A_2$, defined by
\begin{equation} \label{eqn:PrecA1_3D}
 A_1= a_{1}^{0} \Delta_1 \otimes I_2 \otimes I_3 + a_{2}^{0} I_1 \otimes \Delta_2 \otimes I_3 + 
 a_{3}^{0} I_1 \otimes I_2 \otimes \Delta_3, 
\end{equation}
with the scaling constants, $a_{\ell}^{0} = \sum\limits^{R}_{k=1}a_{\ell,k}^{0} \prod\limits_{m=1; m\neq \ell}^d d_{m,k}^0 \ , 
\; \ell=1,2,3$, and 
\begin{equation} \label{eqn:PrecA2_3D}
 A_2= A_1^0 \otimes I_2\otimes I_3 + I_1 \otimes A_{2}^{0}\otimes I_3 +
 I_1 \otimes I_2 \otimes A_{3}^{0},
\end{equation}
where the matrices $A_\ell^0\in \mathbb{R}^{n\times n}$ are defined by
\[
 A_1^0 = \sum\limits^{R}_{k=1} d_{2,k}^0 d_{3,k}^0 A_{1,k}, \quad
 A_2^0 = \sum\limits^{R}_{k=1} d_{1,k}^0 d_{3,k}^0 A_{2,k}, \quad 
 A_3^0 = \sum\limits^{R}_{k=1} d_{1,k}^0 d_{2,k}^0 A_{3,k},
\]
respectively. 
 In 3D case the asymptotic complexity of preconditioner $B_2$ is close to $O(d \, n \log \, n)$
since the ``complexity term'' $C_1\, n^2$ appears with a small constant
and remains negligible in the practically interesting range of the univariate grid-size parameter $n$.

\begin{lemma}\label{lem:prec_4_A}
 Let the matrices $A_1$ and $A_2$ be given by (\ref{eqn:PrecA1_3D}) and (\ref{eqn:PrecA2_3D}), 
 respectively, and the diagonal elements in (\ref{eqn:weights}) satisfy
 \[
   0 < d^{-}_{\ell,k} \leq {\bf d}_{\ell,k}[i]\leq d^{+}_{\ell,k}, \quad i=1,\ldots,n.
 \]
 Then the following spectral equivalence estimates hold
 \begin{equation} \label{eqn:PrecA1_3D_bounds}
   (1-q_A)(1-q_D)^{d-1}  A_1  \leq A  \leq (1+q_A)(1+q_D)^{d-1} A_1,
 \end{equation}
 where 
\[
q_A=\max\limits_
{\ell,k}\frac{a^{+}_{\ell,k}-a^{-}_{\ell,k}}{a^{+}_{\ell,k}+a^{-}_{\ell,k}} <1,
\]
does not depend on $n$. Furthermore, we have
\begin{equation} \label{eqn:PrecA2_3D_bounds}
    (1-q_D)^{d-1} A_2 \leq A  \leq (1+q_D)^{d-1} A_2,
\end{equation}
where 
\[
q_D=\max\limits_
{\ell,k}\frac{d^{+}_{\ell,k}-d^{-}_{\ell,k}}{d^{+}_{\ell,k}+
d^{-}_{\ell,k}} <1
\]
does not depend on $n$.
\end{lemma}
\begin{proof}
 We split our proof into two arguments. To that end, we represent the averaged matrix 
 $A_2$ in the equivalent form
\begin{equation} \label{eqn:3D_oper_average}
A_2 = \sum\limits^{R}_{k=1}  \left(A_{1,k} \otimes  d_{2,k}^0 I_{2} \otimes d_{3,k}^0 I_{3}+  
d_{1,k}^0 I_{1} \otimes A_{2,k} \otimes d_{3,k}^0 I_{3} + 
d_{1,k}^0 I_{1} \otimes d_{2,k}^0 I_{2} \otimes A_{3,k} \right),
\end{equation}
where the constants $d_{\ell,k}^0$ denote the mean-value of the corresponding diagonals $D_{\ell,k}$, 
see (\ref{eqn:weights}). Now we apply the simple spectral bounds
\[
  (1-q_D) d_{\ell,k}^0 I_{\ell}  \leq D_{\ell,k} \leq (1+q_D) d_{\ell,k}^0 I_{\ell},\quad \ell=1,2,3,
  \quad k=1,\ldots,R,
\]
to obtain (for $\ell=1$)
\[
(1-q_D)^{d-1}A_{1,k} \otimes  d_{2,k}I_{2} \otimes d_{3,k}I_{3}
     \leq A_{1,k} \otimes  D_{2,k} \otimes D_{3,k} \leq 
     (1+q_D)^{d-1}A_{1,k} \otimes  d_{2,k}I_{2} \otimes d_{3,k}I_{3},
\]
and similar for $\ell=2,3$. Summing up the above bounds over $\ell=1,2,3$
and $k=1,\ldots,R$, we arrive at the claim (\ref{eqn:PrecA2_3D_bounds}) for $A_2$.

To justify the estimate (\ref{eqn:PrecA1_3D_bounds}) we combine the previous argument 
with the corresponding spectral bounds for the tridiagonal matrices 
$A_{\ell,k}$ in terms of the constant $q_A$,
\[
 (1-q_A)\Delta  \leq A_{\ell,k}\leq (1+q_A)\Delta, \quad \ell=1,2,3; \quad k=1,
 \ldots,R,
\]
which completes the proof.
\end{proof}

 Following the arguments similar to those in \cite{SchmittKh2Sch:20},
the condition number for both preconditioners $B_1^{-1}$ and $B_2^{-1}$,  as well as for $S_1^{-1}$ and $S_2^{-1}$,
applied to the system matrix $B$ and $S$, respectively, 
can be estimated in terms of given quantities  
$d_{\ell,k}^0$ and $a_{\ell,k}^{0}$ and the corresponding equivalence constants for each of 
$R$-terms in (\ref{eqn:3D_operator}). This issue is considered in the following sections.

\subsubsection{Analysis of preconditioners for the system matrix $\mathbf{B= A +A^{-1}}$ }

Taking into account Lemma \ref{lem:prec_4_A} the condition number for the preconditioned 
matrix $B= A +A^{-1} $ can be estimated as follows
\begin{theorem}\label{thm:cond-A1pAm1}
Let the preconditioners be given by $$B_1=A_1+A_1^{-1} \qquad \text{and} \qquad B_2=A_2+A_2^{-1},
$$ 
where the generating matrices $A_1$ and $A_2$ are defined by 
(\ref{eqn:PrecA1_3D}) and (\ref{eqn:PrecA2_3D}), respectively.
Then under the above assumptions on the equation coefficients the following spectral bounds hold
\begin{equation} \label{eqn:B2_for_caseI}
 \mbox{cond}(B_2^{-1} B)  \leq 
 \frac{\max\{(1+q_D)^{d-1},(1-q_D)^{1-d}\}}{\min\{(1+q_D)^{1-d},(1-q_D)^{d-1}\}}\equiv Q^{d-1}, 
\end{equation}
and
\begin{equation} \label{eqn:B1_for_caseI}
 \mbox{cond}(B_1^{-1} B) \leq 
 \frac{\max\{(1+q_A)(1+q_D)^{d-1},(1-q_A)^{-1} (1-q_D)^{1-d}\}}
 {\min\{(1+q_A)^{-1}(1+q_D)^{1-d},(1-q_A)(1-q_D)^{d-1}\}},
\end{equation}
for $d=2,3,$ uniformly in $n$.
\end{theorem}
\begin{proof}
Based on Lemma \ref{lem:prec_4_A} we derive the spectral bounds
\[
      (1+q_D)^{1-d}  A_2   \leq A^{-1}\leq (1-q_D)^{1-d}  A_2,
\]
and 
\[
   (1+q_A)^{-1}(1+q_D)^{1-d}  A_1   \leq A^{-1}\leq (1-q_A)^{-1}(1-q_D)^{1-d}  A_1,
\]
then the results follow by combining the above estimates with (\ref{eqn:PrecA2_3D_bounds}) and 
(\ref{eqn:PrecA1_3D_bounds}), respectively.
\end{proof}

\subsubsection{Analysis of preconditioners for the system matrix $\mathbf{S= A^2 +I} $ }

In view of Lemma \ref{lem:prec_4_A}  the analysis of preconditioner for the modified matrix  $S=A^2+I$ 
can be done along the same line.
To that end, we use the preconditioners $S_1=A_1^2 +I$ and $S_2=A_2^2 +I$, where the generating 
matrices $A_1$ and $A_2$ are defined as above. The following theorem holds.
\begin{theorem}\label{thm:cond-AAp1}
Let the preconditioners be given by 
$$S_1=A_1^2 +I  \qquad \text{and} \qquad S_2=A_2^2 +I,
$$ 
where the generating matrices $A_1$ and $A_2$ are defined by 
(\ref{eqn:PrecA1_3D}) and (\ref{eqn:PrecA2_3D}), respectively.
Then under the above assumptions on the equation coefficients the following condition 
number bounds hold
\begin{equation} \label{eqn:S2_for_caseII}
 \mbox{cond}(S_2^{-1} S)  \leq  \frac{(1+q_D)^{2(d-1)}}{(1-q_D)^{2(d-1)}}, 
\end{equation}
and
\begin{equation} \label{eqn:S1_for_caseII}
 \mbox{cond}(S_1^{-1} S) \leq 
 \frac{(1+q_A)^2(1+q_D)^{2(d-1)}}{(1-q_A)^2(1-q_D)^{2(d-1)}},
\end{equation}
for $d=2,3,$ uniformly in $n$.
\end{theorem}
\begin{proof}
Lemma \ref{lem:prec_4_A} ensures the following spectral bounds
\[
 (1-q_A)^2(1-q_D)^{2(d-1)} (A_1^2 +I) \leq   A^2 + I \leq (1+q_A)^2(1+q_D)^{2(d-1)} (A_1^2 +I),
\]
and
\[
  (1-q_D)^{2(d-1)} (A_2^2 +I) \leq   A^2 + I \leq (1+q_D)^{2(d-1)} (A_2^2 +I),
\]
then the result follows.
\end{proof}

Theorem  \ref{thm:cond-AAp1} proves that the condition number estimate (\ref{eqn:S2_for_caseII}) 
for preconditioner $S_2$ provides stronger bound than that in (\ref{eqn:S1_for_caseII}) for $S_1$ in view of the estimate
\[
 \frac{(1+q_A)^2}{(1-q_A)^2} >1.
\]

\subsubsection{Comparison of preconditioners for two alternative formulations }

Now we are able to compare the condition numbers for the preconditioned equations (\ref{eqn:Control}) 
and (\ref{eqn:Control_mod}). The analysis of the right-hand side in 
equation (\ref{eqn:B2_for_caseI}) (in the following denoted by $Q^{d-1}$) for the preconditioner $B_2$ 
leads to the following three cases depending on respective value of $\max$ and $\min$ 
in the numerator and denominator:
\[
 Q^{d-1}= (1+q_D)^{2(d-1)}, \quad Q^{d-1}= \frac{(1+q_D)^{d-1}}{(1-q_D)^{d-1}}, \quad Q^{d-1}= \frac{1}{(1-q_D)^{2(d-1)}}.
\]
This shows that in all cases we may expect the relations
\[
     \mbox{cond}(B_2^{-1} B)    <   \mbox{cond}(S_2^{-1} S).
\]
The similar calculations lead to the estimate
\[
     \mbox{cond}(B_1^{-1} B)    <   \mbox{cond}(S_1^{-1} S).
\]

We conclude that in all cases the PCG iteration for solving the initial equation 
(\ref{eqn:Control}) is expected to converge faster than in the case of modified formulation (\ref{eqn:Control_mod}).
On the other hand, the rank truncated matrix-vector multiplication with the system matrix
$A^2+I$ can be implemented much faster by using the simple Horner scheme for multiplication of matrix
polynomial with a vector, 
$$A^2 u = A(A u),
$$
than that in the case with the system matrix $A + A^{-1}$. Indeed, in the latter case one needs the 
embedded PCG iteration to solve the equation $A y= u$, required to calculate the action 
$A^{-1} u$ at each iteration step. 

 Hence the comparison of numerical complexity for two equivalent formulations 
(\ref{eqn:Control}) and (\ref{eqn:Control_mod})
is not a straightforward task in general since the trade-off 
between the reduced number of global PCG iterations vs. the reduced cost of matrix-vector operations, 
(i.e. the cost of one PCG iteration) mainly depends on the particular problem setting. 

\subsection{Enhancing the PCG iteration }\label{ssec:speedupPCG} 

In this section, we emphasize several issues which can be employed to optimize the convergence of the PCG iteration and to reduce 
the numerical cost of the whole solution process.

\begin{itemize}

\item[I.] With regard to the control of the number of PCG iterations, the optimal balance between the rank-truncation 
error $\varepsilon_{trunc}$ and stopping criteria 
$\varepsilon_{PCG}$ in the PCG iteration is important. In our numerical tests we therefore use the relation 
$\varepsilon_{trunc} \leq  C_0\varepsilon_{PCG} $ with $C_0\in [0.1,0.01]$.

\item[II.] Using the {\it cascading} multigrid \cite{Shaidur:95,BornDeuf:1996} may reduce the iteration number on fine grids. 
The benefits of this approach are the following: 

(a) There is only low extra cost for realization of the multigrid scheme compared with the single-grid one; 

(b) The interpolation operator from coarse to fine grids applies only to 1D data (vectors) since 
the multivariate functions (grid-based solutions) are represented in the rank-structured 
separable form; 

(c) The number of PCG iterations on the finest grid  can be substantially reduced thanks to 
the good initial guess interpolated from coarser grid $\Omega_h$ to finer grid $\Omega_{h/2}$,
\[
 u_h \mapsto u_{h/2} 
\]
by using linear or cubic splines. 

\item[III.] The practical control of the FEM/FDM approximation in the considered control problems is a nontrivial task.
Estimating the convergence rate of the FEM/FDM discretization in terms of the 
mesh-size can be performed on a sequence of grids, thus evaluating the true scale for the stopping criteria
(in order to avoid over-iteration). We assume that the FEM/FDM approximate solution on the grid $\Omega_h$
can be represented for small enough mesh-size $h$ in the form
\begin{equation}\label{eqn:error_decay}
 u_h = \hat{u}_h^\ast + \hat{C}_1 h^\alpha + O(h^{\alpha+1}), \quad \alpha>0,
\end{equation}
where the discrete functions $\hat{u}_h^\ast$ and $\hat{C}_1$ are the traces 
on the grid $\Omega_h$ of the exact solution 
$u^\ast(x)$ and the continuous function $C_1(x)$, respectively. Here the function $C_1(x)$
does not depend on the mesh parameter $h$.
Then we readily obtain the decomposition of the inter-grid error decay  by 
restricting $u_{h/2}$ onto the coarser grid $\Omega_h$, 
\[
 u_h - u_{h/2} = (1-2^{-\alpha}) \hat{C}_1 h^\alpha + O(h^{\alpha+1}).
\]
Now calculating the ratio on two pairs of refined grids leads to the estimate
\[
 \frac{\| u_{2h} - u_{h}\|_2}{\| u_h - u_{h/2}\|_2} \approx 2^\alpha,
\]
which indicates the decay rate of interest, $\alpha>0$, where the norm in the nominator is calculated for 
the traces on the coarsest grid $\Omega_{2h}$, while denominator 
is evaluated on $\Omega_h$. The estimated parameter $\alpha$ can be 
used in the error representation (\ref{eqn:error_decay}).
In our numerical tests we observe the range of exponential $\alpha\in [3/2;2]$.  In most cases we 
arrive at the value $\alpha=2$ indicating the approximation error of the order of $O(h^2)$ for the 
constructed FDM discretization scheme.

\end{itemize}


\section{Numerical Results} \label{sec:numerics}

In this section, we present numerical results for given 2D data. We consider the target operator $A$ 
to be in a (low-)rank structured form as in \eqref{eqn:2D_operator}. First, we are going to solve the illustrative example equation 
\begin{equation}\label{eqn:eq1}
Au = F 
\end{equation}
with respect to $u$ to show that our low-rank PCG solution scheme making use of the algorithm introduced 
in  \cite{HKKS:18} and presented in the Appendix of this paper indeed serves as a proper solver. 
Subsequently, we aim at solving equation \eqref{eqn:Control} in order to solve control problem \eqref{eqn:DiscrOptimal}. 
In order to avoid embedded iterations, we solve the reformulated version of \eqref{eqn:Control},  namely
\begin{gather}
(\gamma A^2 + I)u = AF \label{eqn:eq2}
\end{gather}
for the control $u$. 

\begin{figure}[H]
 \centering
 	
 	\subfigure[Gaussian RHS $F$]{\includegraphics[width=0.32\textwidth]{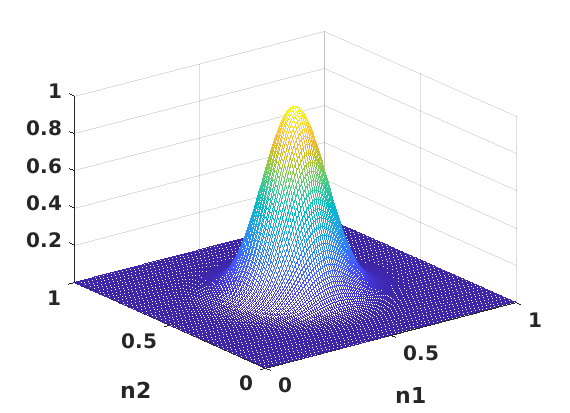}}
 	\subfigure[$AF$ for test 1]{\includegraphics[width=0.32\textwidth]{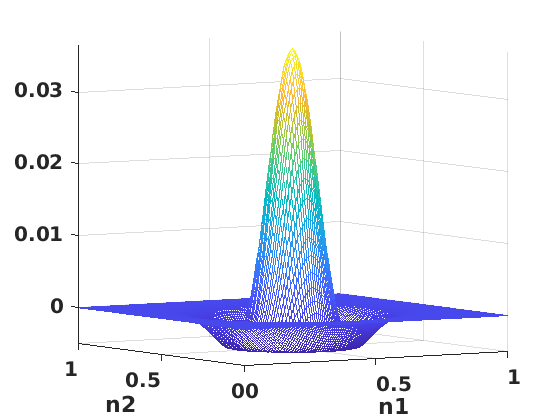}}
 	\subfigure[$AF$ for test 2]{\includegraphics[width=0.32\textwidth]{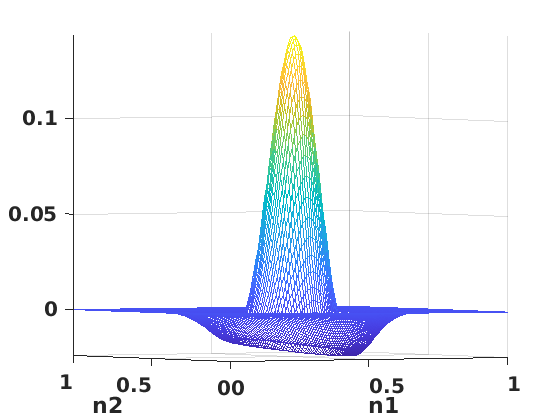}}
 	\caption[Test 2: Lösungen $u$ mit $\alpha = 1$]{Gaussian-type right hand side (RHS) $F$ and 
 	right hand side $AF$ of equations \eqref{eqn:eq1} and \eqref{eqn:eq2}, respectively.}
 	\label{Fig:design2D}
 \end{figure}
For both equations, the optimal design $F$ is given as the rank-1 representation of a Gaussian function, 
see figure \ref{Fig:design2D}.
Ensuring low-rank representations of all involved operators throughout the solution scheme keeps the 
storage complexity comparatively low and enables us to solve equations \eqref{eqn:eq1} and \eqref{eqn:eq2} 
on very fine grids, especially with regard to a possible application to 3D tensor data. This topic is discussed 
in more details at the end of this subsection.  Within the PCG algorithm, we use a 
truncated SVD as rank truncation method and use $\varepsilon_{trunc} = 10^{-8}$ as error threshold for truncation. 
We stop the algorithm whenever the relative residual is small enough, 
that is whenever $\frac{||Res||}{||F||} =  \varepsilon_{PCG} \le 10^{-7}$ holds.

 In the following, we present the results for two tests that differ in the rank-1 separable variable 
 coefficients  $a(x_1,x_2) = \sum_{k = 1}^R \prod_{m=1}^d a_m^k(x_m)$  that are part of the elliptic operator 
 defined in \eqref{eqn:Oper_Rank1}. For the sake of clarity, we choose the coefficients to have rank $R = 3$ 
 throughout our tests, that is they are given by
 \begin{equation}\label{eqn:total_coeff}
   a(x_1,x_2) =\sum_{k = 1}^R a_1^k(x_1)a_2^k(x_2). 
\end{equation}

 In {\it test 1}, we use the coefficients 
 \begin{equation} \label{eqn:coeff_test1}
 a_1^k(x_1) = 1, \qquad a_2^k(x_2) = 1, \qquad k = 1,2,3, 
 \end{equation} 
 resulting in the classic  Laplace operator 
 $$\Delta = R( \Delta_1 \otimes I_2 + I_1 \otimes \Delta_2), \quad R = 3.
 $$
 In {\it test 2}, we choose the coefficient functions to be given by 
 \begin{equation} \label{eqn:coeff_test2}
 \begin{alignedat}{2}
 a_1^1(x_1) &=  x_1 + 2 \ , \qquad && a_2^1(x_2)   = 5x_2^2 +2,  \\
 a_1^2(x_1) &= \sin(x_1)\cos(x_1) + 1 \ , \qquad && a_2^2(x_2)  = 1,  \\
 a_1^3(x_1) &= 1 \ , \qquad && a_2^3(x_2)  = \sin(4\pi x_2) + 2. 
 \end{alignedat}
 \end{equation}
 
 Figure \ref{Fig:total_coeff} provides a visualization for the total multi-dimensional diffusion coefficient \eqref{eqn:total_coeff} 
 that is considered within the numerical examples for \textit{test 2}. 
 This total coefficient $a(x_1,x_2)$ contains the one-dimensional rank-1 coefficients defined in \eqref{eqn:coeff_test2}. 
 
 \begin{figure}[H] 
 \centering
 	\includegraphics[width=0.44\textwidth]{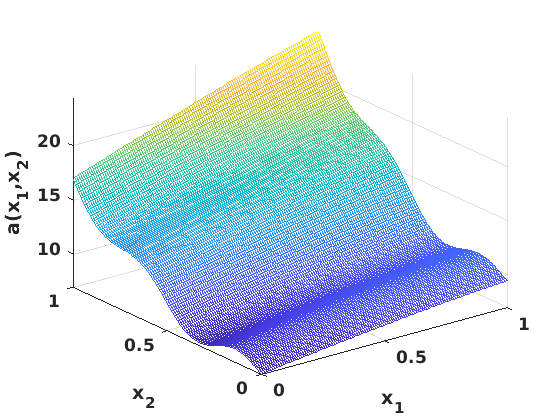}
 	\caption[Test 2: Lösungen $u$ mit $\alpha = 1$]{Total diffusion coefficient $a(x_1,x_2)$ 
 	for \textit{test 2} with single coefficients $a_i^k(x_i), \ {i=1,2,} \ {k = 1,2,3}$ 
 	defined in \eqref{eqn:coeff_test2}. }
 	\label{Fig:total_coeff}
 \end{figure} 
 
 For both test settings and both equations \eqref{eqn:eq1} and \eqref{eqn:eq2}, we present a numerical 
comparison of two different preconditioner  
approaches: When solving \eqref{eqn:eq1}, 
we compare the respective inverse of $A_1$ and $A_2$, that is we 
use $\widetilde{S}_1 = A_1^{-1}$ and $\widetilde{S}_2 = A_2^{-1}$ as preconditioners. 
For equation \eqref{eqn:eq2} we compare the preconditioners $S_1 = (A_1^2 + I)^{-1}$ 
and $S_2 = (A_2^2 + I)^{-1}$, that were introduced and discussed in section \ref{ssec:Precond} 
and where $A_1$ and $A_2$ are respectively defined as in 
\eqref{eqn:PrecA1_2D} and \eqref{eqn:PrecA2_2D} 
via
\begin{equation*} 
 A_1= a_{1}^{0} \Delta_1 \otimes I_2  + a_{2}^{0} I_1 \otimes \Delta_2,
\end{equation*}
\begin{equation*} 
 A_2= A_1^0 \otimes I_2 + I_1 \otimes A_{2}^{0},
\end{equation*}
where
\[
 A_1^0 =  d_{2}^0  A_{1}, \quad
 A_2^0 =  d_{1}^0  A_{2}, \quad 
\]
anisotropy factors $a_\ell^0$ given by \eqref{eqn:ani_factors} and weights $d_{\ell}$ 
given by \eqref{eqn:weights}. \\
All preconditioning operators $S_1$ and $S_2$, $\ \widetilde{S}_1, $ and $\widetilde{S}_2$ 
are represented 
in low-rank structure where we choose the fixed rank parameter $R_{precond} = 10$ throughout all tests.

\subsection{Solving $\mathbf{Au = F}$ and  $\mathbf{(A^2+I)u = AF}$} \label{subsec:Num_Control_Mod}

In this section, we present the
 numerical results for 
the solution of equations \eqref{eqn:eq1} and \eqref{eqn:eq2} for an increasing number of grid points 
and the parameter $\gamma = 1$. The coarsest $n \times n$ grid 
consists of $n^2 = (2^5 -1)^2 = 31^2 = 961 $ grid points, whereas the finest grid consists 
of $n^2 = (2^{12} - 1)^2 = 4095^2 = 16.769.025$ grid points. 

First, we want to verify the usage of our low-rank pcg solver by comparing the time needed to 
solve equation \eqref{eqn:eq1}, $Au = F,$ by the low-rank pcg scheme to the time the direct 
backslash Matlab solver needs, that makes use of a Cholesky factorization. Test setting 1 is used, 
resulting in the classic Laplace operator. The Matlab solver is applied to the direct finite difference 
discretization of equation \eqref{eqn:eq1}, that results in huge system matrices. It is worth to note 
that only by making use of sparse matrix structures it is possible to solve such huge systems 
up to a grid size of $2047^2 = 4190209$ grid points with the direct Matlab solver. 
Table \ref{tab:MatlabSolver} shows the results of the comparison of the aforementioned solvers. 
It can be seen that for growing grid sizes, the rank-structured pcg solver clearly outperforms the 
direct Matlab solver with respect to time needed to solve \eqref{eqn:eq1}. What is more, the 
low-rank pcg scheme ensures a consistent grid independent low rank of the solution operator. 
We therefore underline that ensuring low-rank structures of all involved operators within the pcg scheme 
makes a solution of 
\eqref{eqn:eq1} on very fine grids possible. It was not possible to solve the problem for grid sizes 
exceeding $2047^2$ grid points with the direct Matlab solver on the above mentioned machine, while our 
low-rank pcg scheme could handle the respective large grid size of $4095^2$ grid points. 
For more details on the required storage complexity that is linked to the solution rank, 
see the discussion at the end of this subsection.

In the following, we are going to analyze the performance of the low-rank pcg solution scheme for solving 
the problem-specific control equation \eqref{eqn:eq2}, that is $$(\gamma A^2 + I)u = AF. $$
Table \ref{tab:test1MG2} represents the results for solving equation \eqref{eqn:eq2} when using the 
coefficients defined for test 1, given by \eqref{eqn:coeff_test1} and resulting in the operator $A$ 
to be the classic Laplacian. Note that since all coefficients are chosen to be 1 in this case, 
both preconditioning approaches lead to the same 
preconditioning operator, that is $S_1 = S_2$ since in this case it 
holds $a_1^0 = a_1^0 = d_1^0 = d_2^0 = 1.$ Therefore, the results for both preconditioner styles 
are the same and there is no respective distinction in the table. 

\begin{table} 
\begin{center}
 \begin{tabular}
[c]{c||c|c||c}%
  grid   size   &  time pcg &   sol. rank  &  time   DM \\
  \hline
 $127^2$    & 0.0106 & 4  & 0,0375       \\
     \hline 
 $255^2$	& 0.0051	& 4 & 0,1271 		\\
 	\hline
 $511^2$	& 0.0152	 & 4 & 0,5363		\\
 	\hline
 $1023^2$	& 0.0250	 & 4 & 2,0991		\\
 	\hline
 $2047^2$	& 0.0385	 & 4 & 9,4740		\\
 	\hline
 	$4095^2$	& 0.1069	 & 4 & --		\\
  \end{tabular}
 \caption{Total times in seconds needed by the rank-structured scheme and the direct Matlab (DM) 
 solver to solve \eqref{eqn:eq1} with different numbers of grid points. 
 The setting for test 1 and preconditioner $\widetilde{S_1}$ is used.}
 \label{tab:MatlabSolver}
  \end{center}
  \end{table}
Table \ref{tab:test2MG2} represents the results for solving equation \eqref{eqn:eq2} using the coefficients 
defined for test 2, defined in \eqref{eqn:coeff_test2}.
We show the differences in the results 
for different kinds of preconditioners $S_1$ and $S_2$ as described 
in the introductory part of this chapter. As theoretically shown in subsection \ref{ssec:Precond}, 
Theorem \ref{thm:cond-AAp1}, 
we note that the numerical results confirm the superiority of a preconditioning operator in the style 
of $A_2 $ in contrast to using $ A_1$ since we notice a lower number of iterations needed to 
solve the problem, which is also associated to a faster convergence of the PCG algorithm. 

Tables \ref{tab:test1MG2} and \ref{tab:test2MG2} both verify that the number of iterations 
needed by the pcg scheme to solve equation \eqref{eqn:eq2} does not (test 1) or only slightly (test 2) 
changes when a mesh refinement is done, implying that a successive performance of the low-rank pcg solver 
is ensured independently of the univariate grid size. 
\begin{table}[h] 
\begin{center}
\begin{tabular}[h]{c|c|c|c}
grid size & \# iter & time pcg (in sec.) &  sol. rank \\
\hline
$31^2$	& 1		& 0.0037		& 4	\\
$63^2$ 	& 1   	& 0.0047   	& 4   \\
$127^2$ 	& 1   	& 0.0092    	& 4   \\
$255^2$ 	& 1   	& 0.0055  	& 4 \\
$511^2$ 	& 1   	& 0.0113  	& 4   \\
$1023^2$ 	& 1   	& 0.0309   	& 4   \\
$2047^2$ 	& 1   	& 0.0792    	& 4   \\
$4095^2$	& 1		& 0.3479		& 4	\\
\end{tabular}
\caption{Results for solving equation \eqref{eqn:eq2} when using coefficients for test 1 and $\gamma = 1$. 
A cascadic mesh refinement is done.}
 \label{tab:test1MG2}
\end{center}
\end{table}
Another important observation is that the rank of the solution matrix $u$ does not (test 1) or only observable
slightly (test 2) change over different grid sizes, analogously to the case when solving test 
equation \eqref{eqn:eq1} (see table \ref{tab:MatlabSolver}). That fact ensures that by using 
low-rank structures within the solution scheme, equations \eqref{eqn:eq1} and \eqref{eqn:eq2} 
can be solved for a huge number of grid points. Working with the full-format operator matrix $A$ 
would require the storage of a $n^2 \times n^2$ matrix, which is only tractable when exploiting 
sparse structures. As previously shown in table \ref{tab:MatlabSolver}, this storage requirement 
cannot be handled by the used laptop (16GB RAM) when a grid size of $2047^2$ grid points is exceeded. 
However, when using a low-rank decomposition of the respective operator, the storage can be reduced 
dramatically, resulting in a storage cost of the order $O(dRn),$ where in our tests $d=2$ and $R$ 
denotes the rank of the solution $u$ when solving \eqref{eqn:eq1} or \eqref{eqn:eq2}. 
The aforementioned fact is the reason why it is important to ensure low-rank structures throughout the algorithm. 
The low-rank decomposition can be computed by using svd (in the 2D case) or the canonical tensor format (in the 3D case). 
For a more detailed discussion on this topic, also see the recent paper of the authors \cite{SchmittKh2Sch:20}.
 Analyzing the time the pcg scheme needs in order to solve equation \eqref{eqn:eq2} 
for different univariate grid sizes $n$, we conclude that a nearly linear complexity scaling can be 
observed for both \textit{test 1} and \textit{test 2}. 
Figure \ref{Fig:control2D} represents the control solution $u$ of the problem-specific control 
equation \eqref{eqn:eq2} for two different test settings when using preconditioner style $S_2$, 
regularization parameter $\gamma = 1$ and using a grid of size $n_1 = n_2 = 255$ resulting in a total 
number of $n^2 = 255^2$ grid points. 

\begin{table}[h] 
\begin{center}
\begin{tabular}[h]{c|c c|c c|c c|c c}
grid size & \multicolumn{2}{c|}{\# iter} & \multicolumn{2}{c|}{time pcg (in sec.)} & \multicolumn{2}{c|}{time per iter} & \multicolumn{2}{c}{sol. rank} \\
 & $S_1$ & $S_2$ & $S_1$ & $S_2$ & $S_1$ & $S_2$ & $S_1$   & $S_2$\\
\hline
$31^2$	& 8 & 5		& 0.0328 &	0.0116	& 0.0041		& 0.0023	& 12 & 10	\\
$63^2$ 	& 9 & 5  	& 0.0581 &   0.0204 	& 0.0065 	& 0.0041		& 10 & 10 \\
$127^2$	& 9 & 5 		& 0.0854 &   0.0513 	& 0.0095 	& 0.0103		& 11 & 9 \\
$255^2$ 	& 9 & 5 		& 0.2150 &   0.0937 	& 0.0239 	& 0.0187		& 10 & 10\\
$511^2$	& 9 & 5 		& 0.3159 &   0.1719 	& 0.0351 	& 0.0344		& 10 & 10\\
$1023^2$  	& 9 & 5 		& 0.5551 &   0.2763 	& 0.0617 	& 0.0553		& 11 & 10 \\
$2047^2$  	& 9 & 5 		& 1.0924 &  0.7116	& 0.1214 	& 0.1423		& 10 & 10 \\
$4095^2$		& 9	& 5		& 2.7344 &	2.3951 & 0.3038	& 0.4790	& 10 & 10 \\
\end{tabular}
\caption{Results for solving equation \eqref{eqn:eq2} when using coefficients for test 2 and $\gamma = 1$. 
Two different preconditioner styles $S_1$ and $S_2$ are considered and a cascadic mesh refinement is done.}
\label{tab:test2MG2}
\end{center} 
\end{table} 
\vspace{-0.2cm}
\begin{figure}[H] 
 \centering
 	\subfigure[Solution $u$ for test 1]{\includegraphics[width=0.42\textwidth]{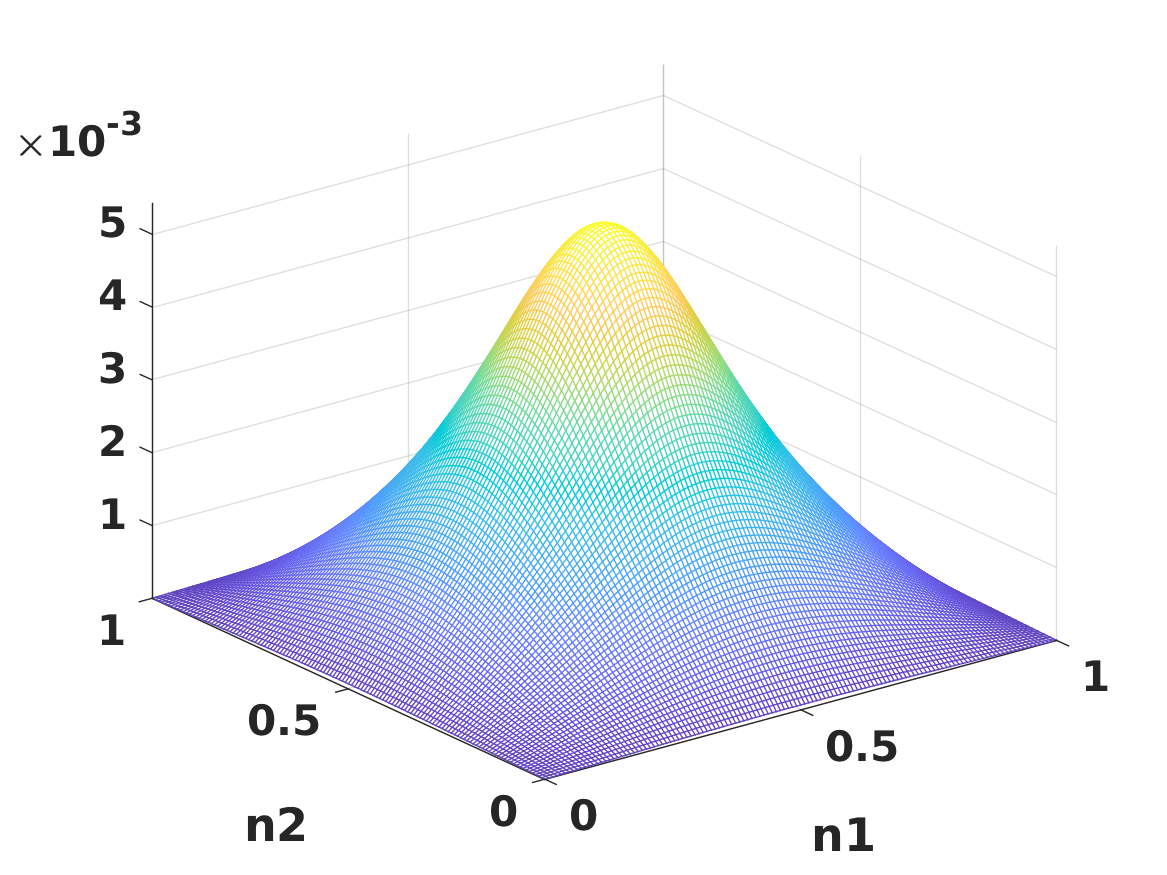}}
 	\subfigure[Solution $u$ for test 2]{\includegraphics[width=0.42\textwidth]{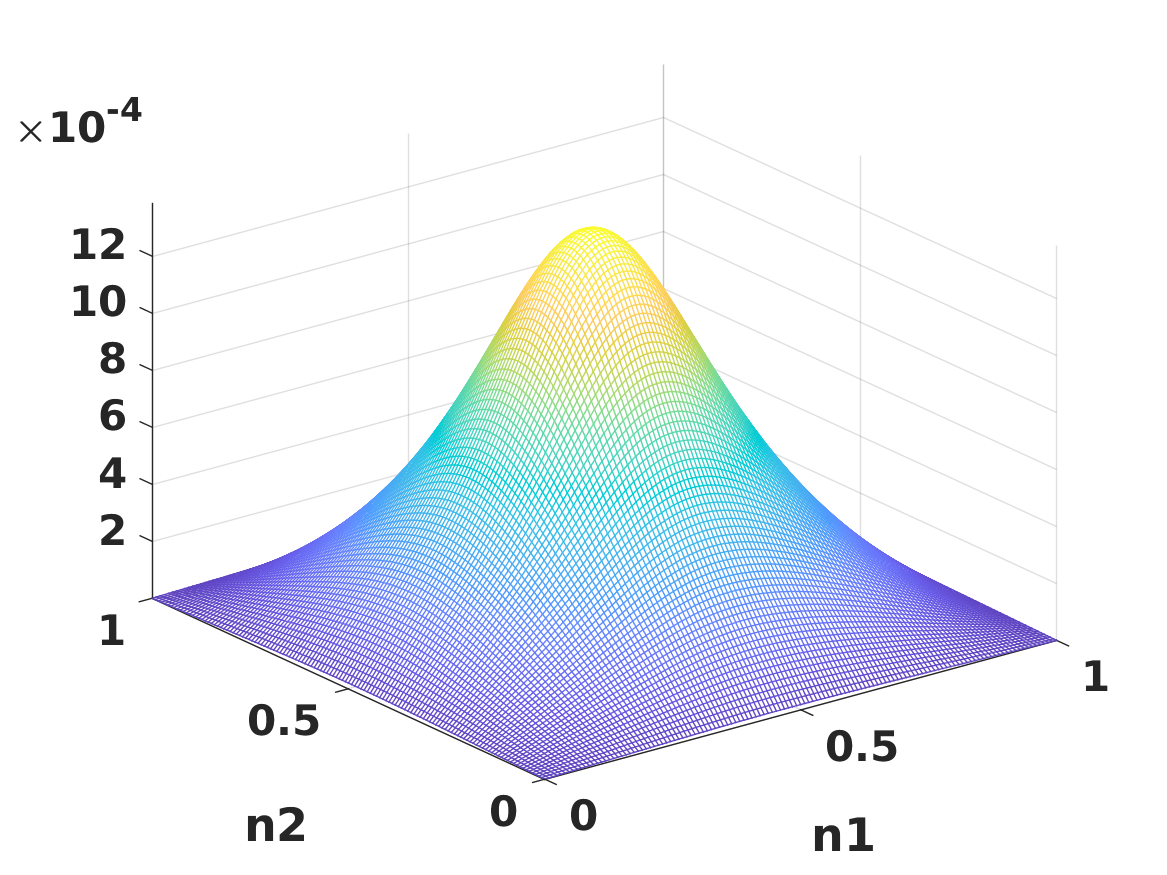}}
 	\caption[Test 2: Lösungen $u$ mit $\alpha = 1$]{Solution $u$ of problem-specific control 
 	equation \eqref{eqn:eq2} for different test settings 1 (left) and 2 (right), 
 	preconditioner $S_2$, $\gamma = 1$, and  $n^2 = 255^2$  grid points.}
 	\label{Fig:control2D}
 \end{figure} 
\begin{remark} In general, for the solution of optimal control problems of type \eqref{eqn:cost_func}, 
the regularization parameter $\gamma$ is chosen to be very small, $\gamma \ll 1$, putting less weight on 
the regularization part of the objective function. However, our solution approach to the respective 
optimal control problem requires the solution of the control equation \eqref{eqn:eq2}. For this equation, 
numerical tests indicate that the smaller parameter $\gamma$ is chosen, the smaller the computational effort 
gets in order to solve equation \eqref{eqn:eq2}. This is due to the fact that for $\gamma \rightarrow 0$, 
equation \eqref{eqn:eq2} approaches $$(N + I)u = F,$$
where $ N := \gamma A^2$ and   $N \stackrel{\gamma \rightarrow 0 }\longrightarrow 0.$
\end{remark}

\subsection{Tests on the discretization convergence rate}

In this subsection, we aim at identifying a explicit convergence rate for the FDM discretization scheme 
we use in our numerical test. As discussed in subsection \ref{ssec:speedupPCG}, we want to estimate the 
rate in terms of the mesh-size, where accordingly to the statements in \ref{ssec:speedupPCG}, 
we calculate the ratio on two pairs of solutions $u$ on refined grids $\Omega_{2h}, \ \Omega_{h},$ and $\Omega_{h/2}$ as 

\begin{equation}
c_h := \frac{\| u_{2h} - u_{h}\|_2}{\| u_h - u_{h/2}\|_2},
\end{equation}
where we expect $c_h \approx 2^\alpha$.

Note that $u$ denotes the solution of equation \eqref{eqn:eq2} when $\gamma = 1$.
For both \textit{test 1} and \textit{test 2}, we compute the ratio $c_h$ numerically for grid sizes of size $n^2 = (2^L-1)^2$ for $L = 6,\dots,11$, implying $h = 1/n$ for $\Omega = [0,1]^2$. 
For the sake of comparison, solutions $u_{2h}$ and $u_{h/2}$ are interpolated onto the grid specified by 
the discrete solution $u_h$ when computing the ratio $c_h$.
Table \ref{tab:convrate} displays the intergrid error ratio for the solutions $u_{2h}, u_h, u_{h/2}$ on refined meshes. 

The numerical results indicate a decay rate of $\alpha = 2 > 0,$ leading to the estimate $c_h \approx 2^2$, 
 which justifies the FDM approximation error of the order of $O(h^2)$.

\begin{table}[H] 
\begin{center}
\begin{tabular}[h]{c|c c c c c c} 
grid size $n^2$ & $63^2$ & $127^2$ &  $255^2$ &		$511^2$ & $1023^2$ & $2047^2$		 \\
\hline
\hline
$c_h$ test 2 $S_2$ & 3,98 & 4,00 & 4,00 & 3,99 & 3,98 & 3,93     \\
$c_h$  test 1 & 4,02 & 4,00 & 4,00 & 3,99 & 4,00 & 3,99
\end{tabular}
\caption{Intergrid error ratio $c_h$ for both \textit{test 1} and \textit{test 2}, where in \textit{test 2} preconditioner style $S_2$ is used.}
\label{tab:convrate}
\end{center}
\end{table}

\subsection{Effects of a cascadic multigrid approach} \label{ssec:MG}

In this subsection, we want to investigate the impact of using a cascading multigrid scheme within our numerical pcg 
solution scheme. Therefore, we use the 1D solution data from grid $\Omega_{h}$ in order to compute an initial 
guess for the refined grid $\Omega_{h/2}$, 
which is done by using a numerically cheap one-dimensional cubic spline prolongation operator
$$p: \R^{\Omega_{h}} \rightarrow \R^{\Omega_{h/2}}$$
that defines a polynomial of third degree.

Table \ref{tab:MG} shows the impact of using the cascading multigrid approach when solving equation \eqref{eqn:eq2} 
within the setting of \textit{test 2} when $\gamma = 1$.
Analyzing the results and comparing them to the results presented in table \ref{tab:test2MG2}, we notice that the number 
of needed iterations in order to solve equation \eqref{eqn:eq2} decreases for increasing grid sizes.\\
On the other hand, concerning the computational complexity, we note that for the finest grid consisting of $4095^2$ grid points, 
the accumulated computational complexity for preconditioner $S_2$ approximately matches, but for preconditioner $S_1$ 
exceeds the time the algorithm needs to solve equation \eqref{eqn:eq2} without using a multigrid approach, 
see the respective \textit{time pcg (in sec.)} in table \ref{tab:test2MG2}. 
The accumulated complexity takes into account the time the algorithm needs in order to solve \eqref{eqn:eq2} 
on all previous grid sizes, leading to the initial guess that is subsequently used as inital guess on the finest grid. 
The dates of table \ref{tab:MG} also show that the time needed for a single iteration increases significantly for fine grids. 
This effect is discussed in details whithin the next section. 

\begin{table}[H] 
\begin{center}
\begin{tabular}[h]{c|c c|c c|c c|c c}
grid size & \multicolumn{2}{c|}{\# iter} & \multicolumn{2}{c|}{time pcg (in sec.)}  & \multicolumn{2}{c|}{time per iter} & \multicolumn{2}{c}{sol. rank}  \\
 & $S_1$ & $S_2$ &  $S_1$ & $S_2$  & $S_1$ & $S_2$ & $S_1$   & $S_2$ \\
\hline
$31^2$		& 8 & 5		& 0.0842  &	  0.0150		& 0.0404			& 0.0030		& 12 & 10	\\
$63^2$ 		& 5 & 3  	& 0.0518  &   0.0232 		& 0.0104 	& 0.0077		& 10 & 10	\\
$127^2$		& 5 & 3 		& 0.0810  &   0.0658		 	& 0.0162 	& 0.0219		& 10 & 10	 \\
$255^2$ 		& 4 & 2 		& 0.2010  &   0.0848		 	& 0.0503 	& 0.0424		& 10 & 10	\\
$511^2$		& 4 & 2 		& 0.3667  &   0.1699		 	& 0.0917		& 0.0850		& 9 & 9		\\
$1023^2$  	& 4 & 2 		& 0.7122  &   0.2902		 	& 0.1781 	& 0.1451		& 9 & 9 		\\
$2047^2$  	& 4 & 2 		& 1.5562  &  0.6638			& 0.3890 	& 0.3319		& 9 & 9		 \\
$4095^2$		& 3	& 1		& 2.0720 &	 1.1479			& 0.6907 	& 1.1479		& 9 & 9 		 \\ \cline{1-5} 
			 \multicolumn{3}{r|}{\multirow[t]{2}{2cm}{\textit{\footnotesize accumulated time}}}	&	\multirow[c]{2}{*}{5.3638}	&   		\multirow[c]{2}{*}{2.4606}					&	\\			
			  \multicolumn{3}{c|}{ }	&		  &				&	 \multicolumn{4}{c}{ }			\\	
\end{tabular} 
\caption{Results for solving equation \eqref{eqn:eq2} when using coefficients for test 2, making use of a cascading multigrid scheme and choosing $\gamma = 1$. Two different preconditioner styles $S_1$ and $S_2$ are considered and a cascadic mesh refinement is done.}
\label{tab:MG}
\end{center} 
\end{table} 

\subsection{Rank propagation}

In the preceding section, more precisely when analyzing table \ref{tab:MG} and comparing it to the 
respective table \ref{tab:test2MG2} that represents the results when no multigrid approach is used, 
we notice a remarkable effect that arises when making use of a multigrid scheme in order to solve 
the problem-specific equation \eqref{eqn:eq2}: While on the one hand, the number of iterations can 
be decreased to a minimum for fine grids, on the other hand, the time in seconds per iteration 
for finer grids when using a multigrid scheme increases 
and clearly exceeds the time per iteration when not using a multigrid scheme
introduced in subsection  \ref{ssec:MG} in order to solve \eqref{eqn:eq2}. 
 Therefore, we notice that 
the cumulative computational complexity of the multigrid scheme, 
obtained by summation the numerical costs over all grid levels, remains merely the same as for the unigrid solver
applied on the finest grid.  

This effect is due to the propagation of the rank of the initial guess operator (vector) $\Xo$ used for 
the pcg scheme that is represented in the appendix.  Please note that we refer to the number of 
columns of the respective operator by its \textit{rank}.  For every tensor-vector multiplication within 
the algorithm, the input tensor  of K-rank $R$ and vector of \cn rank $S$ multiply as 
defined in the multiplication 
scheme \eqref{eqn:multischeme}, resulting in rank $RS$ for any tensor-vector product $Ax$. Furthermore, 
in our particular tests, we are interested in the solution of \eqref{eqn:eq2} that additionally requires 
the computation of a more advanced tensor-vector product $(A^2+ I )u$,  that is computed by making use 
of the Horner scheme, implying a successive multiplication structure $A(Au)$. This method allows for a 
truncation-during-multiplication approach, where the rank $RS$ originating from the first 
multiplication $Au$ is immediately reduced to a rank $r_\varepsilon \ll RS$ with respect to a 
truncation treshold $\varepsilon_{trunc}$. Hence, also taking into account special concatenation 
structures that are used for the numerical tests, the rank of the resulting complete tensor-vector
product $(A^2 + I)u$ can be estimated by $(4Rr_\varepsilon + S)$.  

 As already described in the introduction and the previous paragraph, we make use of a rank reduction 
 procedure $\mathtt{trunc}$ that reduces the rank of the current operator to a given $\varepsilon_{trunc}$-threshold 
 within the multiplication scheme and after every step within {Algorithm 1} in order to ensure 
 trackability throughout the pcg scheme.  In particular, in order to truncate the rank of the aforementioned product, we use a reduced SVD as truncation 
 procedure $\mathtt{trunc}$ that due to considering a full matrix and an accuracy truncation 
 threshold $\varepsilon_{trunc}$ has a computational complexity of order 
  $O(\min(n^2(4Rr_\varepsilon + S), n(4Rr_\varepsilon+S)^2))$ 
 strongly  depending on the rank parameters $R$ and $S$ and therefore representing the most 
 time-consuming part within  the algorithm. 
 
In order to understand the effect explained at the beginning of this subsection, we track the operator rank 
throughout the algorithm for test setting 2, preconditioner $S_2$ and an example grid size 
of  $n^2 = 4095^2$   grid points when a multigrid scheme is used (MG) and when not (noMG). 
In the setting when no multigrid scheme is used, the nullvector $\Xo_{noMG} \in \R^{n\times1}$ 
serves as initial guess  that implies rank $S^{\Xo}_{noMG} = 1.$
When making use of a multigrid scheme, the pcg algorithm computes solutions to \eqref{eqn:eq2} 
starting on the coarsest grid of size $n^2 = 31^2$ up to the finest grid of size $n^2 = 4095^2$, where 
the respective initial guess is computed as described in subsection \ref{ssec:MG}, that is based on 
the solution of the respective previous grid. In our test, that leads to an initial guess 
 $\Xo_{MG} \in \R^{n \times 10}$ that implies rank $S^{\Xo}_{MG} = 10$. Note that in both settings, 
 for the K-rank of the operator $A$ it holds $R_{MG} = R_{noMG} = 3.$
 
Figure \ref{Fig:rankpropagation} represents the rank propagation throughout every step of the pcg 
scheme based on the respective initial guess rank $S^{\Xo}_{noMG}$ and $S^{\Xo}_{MG}$ for test 2, 
preconditioner $S_2$ and grid size $n^2 = 4095^2$. Note that the multigrid scheme requires less iterations 
resulting in less single pcg steps. It is clearly observable that the rank spectrum when using a multigrid 
scheme significantly exceeds the rank spectrum when no multigrid scheme is used, which then also indicates 
a higher computational cost for the respective rank truncation operations by reduced SVD.

\begin{figure}[H] 
 \centering
 	\includegraphics[width=0.44\textwidth]{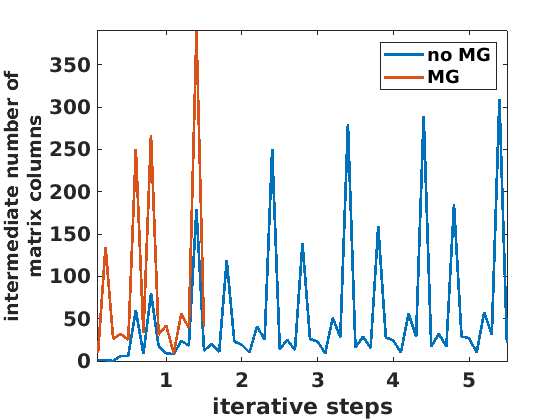}
 	\caption[Test 2: Lösungen $u$ mit $\alpha = 1$]{Rank propagation within the pcg scheme 
 	when using a multigrid 
 	scheme (red line) and when using no multigrid scheme (blue line) for test 2, 
 	preconditioner $S_2$ and grid size  $n^2 = 4095^2$.  }
 	\label{Fig:rankpropagation}
 \end{figure} 

Taking into account the results represented in tables \ref{tab:test2MG2}, \ref{tab:MG}, and 
figure \ref{Fig:rankpropagation}, we can summarize that ensuring low ranks throughout all steps 
of the algorithm, which in turn regulates the computational complexity of the truncated SVD, is the 
most importing aspect in order to keep the total computational cost for solving equation \eqref{eqn:eq2} low. 
Furthermore, the characteristic of the multiplication scheme implies that low ranks are most likely to be 
kept low within the algorithm when the rank of the initial guess, $S^{\Xo}$, is low.
 A comparison of tables \ref{tab:test2MG2} and \ref{tab:MG}  indicates that the accumulated time 
for the multigrid iteration on finest grid is merely the same as the CPU time of the unigrid PCG iteration 
performed on the same grid though the MG algorithm converges in only one iteration.  \\
To conclude, we state that the advantages of ensuring a low rank of the initial guess $\Xo$ and 
subsequently staying on a low rank manifold clearly outweigh the advantages of starting with a 
good initial guess within a  multigrid scheme. It is worth to note that in our numerical tests, 
we also observed this effect when choosing smaller initial guess ranks $1 < S_{MG} \ll 10$.

\section{Conclusion} \label{sec:Conclusions}

In this paper, we introduce a tensor numerical method for the efficient solution of optimal control problems  
constrained by elliptic operators in $\mathbb{R}^d$, $d=2,3$, with variable coefficients, 
which can be represented in the low rank separable form. The equation for the control function is 
discretized on large 
$n^{\otimes d}$ grids and then solved by a PCG iteration with adaptive rank truncation. 
We construct and analyze the family of spectrally equivalent preconditioners by using tensor decomposition 
of matrix valued functions of the
anisotropic Laplacian, which then are diagonalized in the Fourier basis. 

 Our numerical tests confirm that the discussed low rank PCG scheme serves as proper solver 
for the introduced problem class, 
as it outperforms considerably the direct Matlab solver especially on large grids. 
Furthermore, our numerical tests on the solution of the 
problem-specific control equation in 2D setting validate a nearly linear complexity scaling in the  
grid size $n^2$, where $n\times n$ grids of 
size $n^2 = 31^2$ up to $n^2 = 4095$ are considered, while the number of iterations remains also stable 
throughout all grid sizes. 
A remarkable effect is that the $K$-rank $R$ of the solution operator remains unaffected small 
throughout all grid sizes, ensuring a small 
storage complexity of order $O(dRn)$ compared to the storage complexity that would be required 
in order to store the full solution 
operator of size  $n^2 \times n^2 $. Numerical investigations on the discretization convergence 
rate justify a FDM 
approximation error of order $O(h^2)$ for our tests. What is more, all aforementioned results hold 
for both types of 
spectrally equivalent preconditioners that have been introduced in the course of the paper. 
However, the numerical tests validate the 
theoretical findings, stating that the initially defined operator-based preconditioner $S_2$ 
outperforms  the anisotropic 
Laplacian-type preconditioner $S_1$.

 Last, we do an analysis on the integration of a cascadic multigrid method in our PCG scheme. 
 The survey shows that integrating 
 a multigrid scheme does not improve the computational complexity due to increasing ranks of 
 the involved operators, however, 
 the number of iterations can be efficiently decreased for increasing grid sizes. 
 All in all, the results put stress on the 
 importance of ensuring a low-rank structure of all involved quantities throughout the algorithm.

 Finally, we notice that our approach can be generalized to the case of optimal control problems that contain fractional operators in constraints, however this problem class will
be considered elsewhere.

\section*{Appendix}
 We present the PCG iterative algorithm performed in a low-rank canonical format with the adaptive rank truncation 
for solving the linear system of equations for control function in $\mathbb{R}^3$, where the system 
matrix is given in low Kronecker rank format.   
As the rank truncation procedure,
in our implementation we apply the reduced SVD algorithm since we work in a two-dimensional room. The algorithm has originally been introduced in \cite{HKKS:18}.

\begin{algorithm}[H]
	\caption{Preconditioned CG method in low-rank format} \label{alg:pcg}
	\begin{algorithmic}[1]
		\Require{Rank truncation procedure $\mathtt{trunc}$,  procedure $\mathtt{matvec}$ in
		the  low $K$-rank format,  preconditioner in low-rank format
		$\mathtt{precond}$, right-hand side tensor $\mathbf{B}$, initial guess $\mathbf{X}^{(0)}$, rank tolerance parameter $\varepsilon$,
		maximal iteration number $k_{\max}$}
			\State $\mathbf{R}^{(0)} \leftarrow \mathbf{B} - \mathtt{matvec}(\mathbf{X}^{(0)})$
			\State $\mathbf{Z}^{(0)} \leftarrow \mathtt{precond}(\mathbf{R}^{(0)})$
			\State $\mathbf{Z}^{(0)} \leftarrow \mathtt{trunc}(\mathbf{Z}^{(0)},\varepsilon)$
			\State $\mathbf{P}^{(0)} \leftarrow (\mathbf{Z}^{(0)})$
			\State $k \leftarrow 0$
			\Repeat
				\State $\mathbf{S}^{(k)} \leftarrow
				\mathtt{matvec}(\mathbf{P}^{(k)})$
				\State {\color{red}$\mathbf{S}^{(k)} \leftarrow
				\mathtt{trunc}(\mathbf{S}^{(k)},\varepsilon)$}
				\State $\alpha_k \leftarrow
				\frac{\dotprod{\mathbf{R}^{(k)}}{\mathbf{Z}^{(k)}}}
				{\dotprod{\mathbf{P}^{(k)}}{\mathbf{S}^{(k)}}}$
				\State $\mathbf{X}^{(k+1)} \leftarrow
				\mathbf{X}^{(k)} + \alpha_k \mathbf{P}^{(k)}$
				\State {\color{red}$\mathbf{X}^{(k+1)} \leftarrow
				\mathtt{trunc}(\mathbf{X}^{(k+1)},\varepsilon)$}
				\State $\mathbf{R}^{(k+1)} \leftarrow
				\mathbf{R}^{(k)} - \alpha_k \mathbf{S}^{(k)}$
				\State {\color{red}$\mathbf{R}^{(k+1)} \leftarrow
				\mathtt{trunc}(\mathbf{R}^{(k+1)},\varepsilon)$}
				\If {$\mathbf{R}^{(k+1)}$ is sufficiently small}
					\State \Return $\mathbf{X}^{(k+1)}$
					\State \textbf{break}
				\EndIf
				\State $\mathbf{Z}^{(k+1)} \leftarrow
				\mathtt{precond}(\mathbf{R}^{(k+1)})$
				\State {\color{red}$\mathbf{Z}^{(k+1)} \leftarrow
				\mathtt{trunc}(\mathbf{Z}^{(k+1)},\varepsilon)$}
				\State $\beta_k \leftarrow
				\frac{\dotprod{\mathbf{R}^{(k+1)}}{\mathbf{Z}^{(k+1)}}}
				{\dotprod{\mathbf{Z}^{(k)}}{\mathbf{R}^{(k)}}}$
				\State $\mathbf{P}^{(k+1)} \leftarrow
				\mathbf{Z}^{(k+1)} + \beta_k \mathbf{P}^{(k)}$
				\State {\color{red} $\mathbf{P}^{(k+1)} \leftarrow
				\mathtt{trunc}(\mathbf{P}^{(k+1)},\varepsilon)$}
				\State $k \leftarrow k+1$
			\Until{$k = k_{\max}$}
			\Ensure{Solution $\mathbf{X}$ of $\mathtt{matvec}(\mathbf{X})=\mathbf{B}$}
	\end{algorithmic}
\end{algorithm}

\section*{Acknowledgment}
This research has been supported by the German Research Foundation (DFG) within 
the \textit{Research Training Group 2126: Algorithmic Optimization}, Department of Mathematics, Trier University, Germany.

%
%

\begin{footnotesize}

\end{footnotesize}


\begin{thebibliography}{99}

%


\bibitem{Allaire:07} G. Allaire.
 \textit{Numerical analysis and optimization: an introduction
 to mathematical modeling and numerical simulation.}
 Oxford University Press, 2007.
  
%
%

\bibitem{AntilOtarola:2015}
H. Antil and E. Ot\'arola.
\textit{A FEM for an Optimal Control Problem of Fractional Powers of Elliptic Operators.}
SIAM J. Control Optim., 53(6), pp.343-3456, 2015. 


\bibitem{AntilDraGreen:2020}
Harbir Antil, Andrei Dr{\"a}g{\"a}nescu, and Kiefer Green.
\textit{A note on multigrid preconditioning fot fractional PDE-constrained optimization problems.} 
arXiv:2010.14600v1, 2016.



\bibitem{Schwab:18} L. Banjai, J. M. Melenk, R. H. Nochetto, 
E. Ot\'arola, A. J. Salgado and Ch. Schwab.
\textit{Tensor FEM for Spectral Fractional Diffusion.}
Found. Comput. Math., (2018). https://doi.org/10.1007/s10208-018-9402-3.
 

  
   \bibitem{Bonito:18} 
  A. Bonito, J. P. Borthagaray, R. H. Nochetto, E. Otarola, A. J. Salgado.
  \textit{Numerical methods for fractional diffusion}. 
  Computing and Visualization in Science (2018) 1-28.
    
  
  
 \bibitem{BornDeuf:1996} F.A. Bornemann and P. Deuflhard. 
 \textit{The cascadic multigrid method for elliptic problems.} Numer. Math., 75:135-152, 1996.
  
\bibitem{BoSch:2009} A. Borzi and V. Schulz.
\textit{Multigrid methods for PDE optimization.}
SIAM Review 51(2), 2009, 361-395.

\bibitem{BoSch:2012} A. Borzi and V. Schulz.
\textit{Computational optimization of systems governed by partial differential equations.}
Soc. for Ind. and Appl. Math., Philadelphia, 2012.
 
%
%
%
%
%
  
\bibitem{DMV-SIAM2:00}L.~{De Lathauwer}, B.~{De Moor}, J.~Vandewalle.
  \textit{A multilinear singular value decomposition.}
  SIAM J. Matrix Anal. Appl., 21 (2000) 1253-1278.
  
   \bibitem{Reyes:2015} J. C.~{De Los Reyes}.
  \textit{Numerical PDE-Constrained Optimization.}
  SpringerBriefs in Optimization, Springer, Berlin, 2015
  
 \bibitem{DolOs:2011} S. Dolgov and I. V. Oseledets.
 \textit{Solution of linear systems and matrix inversion in the TT-format.}
SIAM J. Sci. Comput. 34 (5), 2011, pp. A2718-A2739.

  
  \bibitem{DPSS:2014} S. Dolgov, J. Pearson, D. Savostyanov and M. Stoll.
  \textit{Fast tensor product solvers for optimization problems with fractional
  differential equations as constraints.}
  Appl. Math. Comp., 273, 2016, 604-623.
    
    
  \bibitem{DuLazPas:18} B. Duan, R. Lazarov and J. Pasciak.
 \textit{Numerical approximation of fractional powers of elliptic operators.}
 arXiv:1803.10055v1, 2018.  
 
  
 
\bibitem{GaHaKh3:02} I. P. Gavrilyuk, W. Hackbusch and B. N. Khoromskij. 
 \textit{Data-Sparse Approximation to  Operator-Valued Functions of Elliptic
   Operator.}  Math. Comp. 73,  (2003), 1297-1324. 
%
%

 
\bibitem {Hack_Book12} W. Hackbusch. 
\textit{Tensor spaces and numerical tensor calculus.}
Springer, Berlin, 2012.
%
  
%
      
\bibitem{HaKhtens:04I} W.~Hackbusch and B. N.~Khoromskij.
  \textit{Low-rank Kronecker product approximation to multi-dimensional nonlocal operators. 
  Part I. Separable approximation of multi-variate functions.}
      Computing {\bf 76} (2006), 177-202.

 \bibitem{HarLazarov:20} 
 S. Harizanov, R. Lazarov, and S. Margenov.
\textit{A survey on numerical methods for spectral space-fractional diffusion problems.}
arXiv:2010.02717v1, 2020.     

%
    
 \bibitem{HKKS:18} G. Heidel, V. Khoromskaia, B. N. Khoromskij and V. Schulz.
 \textit{Tensor product method for fast solution of optimal control problems with 
 fractional multidimensional Laplacian in constraints.} 
 J Comput. Physics, 424, 109865, 2021.
      
      
\bibitem{HerKun:2010} R. Herzog and K. Kunisch.
\textit{Algorithms for PDE constrained optimization.} GAMM, 33 (2010), 163-176.

\bibitem {HaHighTrefeth:08} N. Hale, N. J. Higham, and L. N. Trefethen.
\textit{Computing $A^\alpha$, $\log(A)$, and related matrix functions by contour integrals. }
SIAM J. on Numerical Analysis, 46 (2), 2008, 2505-2523.

\bibitem{HLMMV:2018} S. Harizanov, R. Lazarov, P. Marinov, S. Margenov and Ya. Vutov.
\textit{Optimal solvers for linear systems with fractional powers of sparse spd matrices.}
Preprint arXiv:1612.04846v3, 2018.

\bibitem {Higham_MatrFunc:08} N. J. Higham.
\textit{Functions of Matrices.} SIAM, Philadelphia, 2008.





\bibitem{Khor2Book2:18} V. Khoromskaia and B. N. Khoromskij. 
\textit{Tensor Numerical Methods in Quantum Chemistry.}
Research monograph, De Gruyter Verlag, Berlin, 2018.


  
%

\bibitem{KhorBook:18}
B. N. Khoromskij. \textit{Tensor Numerical Methods in Scientific Computing.}
Research monograph, De Gruyter Verlag, Berlin, 2018.
     
     
  
 
\bibitem{KhKh3:08} B. N. Khoromskij and V. Khoromskaia.
 \textit{Multigrid Tensor Approximation of Function Related Arrays. }     
 SIAM J. Sci. Comp., { 31}(4), 3002-3026 (2009). 

   
 

 
 \bibitem{KreTob:2010} D. Kressner and C Tobler.
\textit{Krylov subspace methods for linear systems with tensor product structure.}
SIAM J Matr. Anal. Appl., 31 (4), pp. 1688-1714.
%
%
 
\bibitem{Kwasn:17} M. Kwa{\'s}nicki. 
\textit{Ten equivalend definitions of the fractional Laplace operator.} 
Functional Calculus and Applied Analysis, 20(1):7-51, 2017.


\bibitem{FracLapl:2018} A. Lischke, G. Pang, M. Gulian, F. Song, 
Ch. Glusa, X. Zheng, Z. Mao, W. Cei, M. M. Meerschaert, M. Ainsworth, G. E. Karniadakis.
\textit{What is the fractional Laplacian?}
arXiv:1801.09767v1, 2018.


\bibitem{SchmittKh2Sch:20} B. Schmitt,  B. Khoromskij, V. Khoromskaia, and V. Schulz.  
\textit{Tensor Method for Optimal Control Problems Constrained by Fractional  
3D Elliptic Operator with Variable Coefficients.} 
E-preprint, arXiv:2006.09314, 2020. 

\bibitem{Shaidur:95} V.V. Shaidurov. \textit{The multigrid methods for finite elements.}
Kluwer Academic Publisher, Dordrecht, 1995.
  
 
 
 \bibitem{Troeltzsch:2005} F. Troeltzsch. \textit{Optimal control of partial 
 differential equations: theory, methods and applications}. 
 AMS, Providence, Rhode Island, 2010.
 
 
 
%
%


\end{thebibliography}
\end{document}